\documentclass[11pt,reqno]{amsart}

\usepackage{amsthm}
\usepackage{amscd}
\usepackage{amsfonts}
\usepackage{amssymb}
\usepackage{amsgen}
\usepackage{amsmath}
\usepackage{amsopn}
\usepackage{verbatim}
\usepackage{xypic}
\usepackage{xspace}
\usepackage{multicol}
\usepackage{url}
\usepackage{upref}

\theoremstyle{plain}
\newtheorem{thm}{Theorem}[section]
\newtheorem{lem}[thm]{Lemma}
\newtheorem{prop}[thm]{Proposition}
\newtheorem{cor}[thm]{Corollary}

\theoremstyle{definition}

\newtheorem{defn}[thm]{Definition}

\theoremstyle{remark}
\newtheorem{rem}[thm]{Remark}

\newtheorem{eg}[thm]{Example}

 \font\cyr=wncyr10
 \newcommand{\nc}{\newcommand}

\DeclareMathOperator{\Inf}{{Inf}}

\DeclareMathOperator{\ind}{{ind}}

\nc{\per}[1]{\underset{#1}{\boldsymbol \pi}\,}

 \nc{\MT}{{\rm MT}}

 \nc{\wht}{{\widehat}}
 \nc{\bwg}{{\bigwedge}}
 \nc{\mmu}{{\boldsymbol{\mu}}}
 \nc{\mal}{{{\scriptstyle \maltese}}}
 \nc{\fA}{{\mathfrak A}}
 \nc{\HH}{{\mathfrak H}}
 \nc{\ra}{\rightarrow}
 \nc{\ors}{{\vec s\,}}
 \nc{\os}{{\overset}}
 \nc{\G}{{\mathbb G}}
 \nc{\Z}{{\mathbb Z}}
 \nc{\R}{{\mathbb R}}
 \nc{\N}{{\mathbb N}}
 \nc{\ZN}{{\mathbb Z_{\ge 0}}}
 \nc{\Q}{{\mathbb Q}}
 \nc{\C}{{\mathbb C}}
 \nc{\Cnn}{{\mathbb C}_{\ge 0}}
 \nc{\Cp}{{\mathbb C}_{>0}}
 \nc{\MPV}{{\mathcal{MPV}}}
 
 \nc{\tB}{{\tilde B}}
 \nc{\Li}{{\rm Li}}
 \nc{\suf}{{\ast\,}}
 \nc{\sufq}{{\ast_q\,}}
 \nc{\gam}{{\gamma}}
 \nc{\gG}{{\Gamma}}
 \nc{\om}{{\omega}}
 \nc{\vep}{{\varepsilon}}
 \nc{\ga}{{\alpha}}
 \nc{\gl}{{\lambda}}
 \nc{\gb}{{\beta}}
 \nc{\gd}{{\delta}}
 \nc{\gs}{{\sigma}}
 \nc{\gS}{{\Sigma}}

 \nc{\gk}{{\kappa}}
 \nc{\tgz}{{\tilde{\zeta}}}
 \nc{\gO}{{\Omega}}
 \nc{\sif}{{\mathcal S}}
 \nc{\gt}{{\tau}}
 \nc{\Lra}{\Longrightarrow}
 \nc{\lra}{\longrightarrow}
 \nc{\fS}{{\mathfrak S}}
 \nc{\DD}{{\mathfrak D}}

 \nc{\Llra}{\Longleftrightarrow}
 \nc{\ol}{\overline}
 \nc{\lms}{\longmapsto}
 \nc{\cv}{{{\mathsf c}{\mathsf v}}}
 \nc{\zq}{{\zeta_q}}
 \nc\qup{{q\uparrow 1}}
 \nc{\us}{\underset}
 \nc{\tn}{{\tilde{n}}}
 \nc{\gD}{{\Delta}}
 \nc{\bi}{{\bf i}}
 \nc{\bfone}{{\bf 1}}
 \nc{\bfa}{{\bf a}}
 \nc{\bfb}{{\bf b}}
 \nc{\bfc}{{\bf c}}
 \nc{\bfd}{{\bf d}}
 \nc{\bfe}{{\bf e}}
 \nc{\bff}{{\bf f}}
 \nc{\bfg}{{\bf g}}
 \nc{\bfi}{{\bf i}}
 \nc{\bfj}{{\bf j}}
 \nc{\bfn}{{\bf n}}
 \nc{\bfl}{{\bf l}}
 \nc{\bfk}{{\bf k}}
 \nc{\bfm}{{\bf m}}
 \nc{\bfo}{{\bf o}}
 \nc{\bfp}{{\bf p}}
 \nc{\bfq}{{\bf q}}
 \nc{\bfr}{{\bf r}}
 \nc{\bfs}{{\bf s}}
 \nc{\bft}{{\bf t}}
 \nc{\bfu}{{\bf u}}
 \nc{\bfv}{{\bf v}}
 \nc{\bfw}{{\bf w}}
 \nc{\bfx}{{\bf x}}
 \nc{\bfB}{{\bf B}}
 \nc{\bfP}{{\bf P}}
 \nc{\bfQ}{{\bf Q}}
 \nc{\bfY}{{\bf Y}}
 \nc{\bfgb}{{\boldsymbol \gb}}
 \nc{\bfga}{{\boldsymbol \ga}}
 \nc{\bfrho}{{\boldsymbol \rho}}
 \nc{\bfchi}{{\boldsymbol \chi}}
 \nc{\QX}{{\Q\langle \bfX\rangle}}
 \nc{\QY}{{\Q\langle \bfY\rangle}}
 \nc{\CX}{{\C\langle \bfX\rangle}}
 \nc{\CY}{{\C\langle \bfY\rangle}}
 \nc{\QXX}{{\Q\langle\!\langle \bfX\rangle\!\rangle}}
 \nc{\QYY}{{\Q\langle\!\langle \bfY\rangle\!\rangle}}
 \nc{\CXX}{{\C\langle\!\langle \bfX\rangle\!\rangle}}
 \nc{\CYY}{{\C\langle\!\langle \bfY\rangle\!\rangle}}

 \nc{\bbA}{{\mathbb A}}
 \nc{\bbB}{{\mathbb B}}
 \nc{\bbC}{{\mathbb C}}
 \nc{\bbD}{{\mathbb D}}
 \nc{\bbE}{{\mathbb E}}
 \nc{\bbF}{{\mathbb F}}
 \nc{\bbG}{{\mathbb G}}
 \nc{\bbH}{{\mathbb H}}
 \nc{\bbI}{{\mathbb I}}
 \nc{\bbJ}{{\mathbb J}}
 \nc{\bbK}{{\mathbb K}}
 \nc{\bbL}{{\mathbb L}}
 \nc{\bbM}{{\mathbb M}}
 \nc{\bbN}{{\mathbb N}}
 \nc{\bbO}{{\mathbb O}}
 \nc{\bbP}{{\mathbb P}}
 \nc{\bbQ}{{\mathbb Q}}
 \nc{\bbR}{{\mathbb R}}
 \nc{\bbS}{{\mathbb S}}
 \nc{\bbT}{{\mathbb T}}
 \nc{\bbU}{{\mathbb U}}
 \nc{\bbV}{{\mathbb V}}
 \nc{\bbW}{{\mathbb W}}
 \nc{\bbX}{{\mathbb X}}
 \nc{\bbY}{{\mathbb Y}}
 \nc{\bbZ}{{\mathbb Z}}
 \nc{\bba}{{\mathbb a}}
 \nc{\bbb}{{\mathbb b}}
 \nc{\bbc}{{\mathbb c}}
 \nc{\bbd}{{\mathbb d}}
 \nc{\bbe}{{\mathbb e}}
 \nc{\bbf}{{\mathbb f}}
 \nc{\bbg}{{\mathbb g}}
 \nc{\bbh}{{\mathbb h}}
 \nc{\bbi}{{\mathbb i}}
 \nc{\bbk}{{\mathbb k}}
 \nc{\bbl}{{\mathbb l}}
 \nc{\bbm}{{\mathbb m}}
 \nc{\bbn}{{\mathbb n}}
 \nc{\bbo}{{\mathbb o}}
 \nc{\bbp}{{\mathbb p}}
 \nc{\bbq}{{\mathbb q}}
 \nc{\bbr}{{\mathbb r}}
 \nc{\bbs}{{\mathbb s}}
 \nc{\bbt}{{\mathbb t}}
 \nc{\bbu}{{\mathbb u}}
 \nc{\bbv}{{\mathbb v}}
 \nc{\bbw}{{\mathbb w}}
 \nc{\bbx}{{\mathbb x}}
 \nc{\bby}{{\mathbb y}}
 \nc{\bbz}{{\mathbb z}}

 \nc{\calA}{{\mathcal A}}
 \nc{\calB}{{\mathcal B}}
 \nc{\calC}{{\mathcal C}}
 \nc{\calD}{{\mathcal D}}
 \nc{\calE}{{\mathcal E}}
 \nc{\calF}{{\mathcal F}}
 \nc{\calG}{{\mathcal G}}
 \nc{\calH}{{\mathcal H}}
 \nc{\calI}{{\mathcal I}}
 \nc{\calJ}{{\mathcal J}}
 \nc{\calK}{{\mathcal K}}
 \nc{\calL}{{\mathcal L}}
 \nc{\calM}{{\mathcal M}}
 \nc{\calN}{{\mathcal N}}
 \nc{\calO}{{\mathcal O}}
 \nc{\calP}{{\mathcal P}}
 \nc{\calQ}{{\mathcal Q}}
 \nc{\calR}{{\mathcal R}}
 \nc{\calS}{{\mathcal S}}
 \nc{\calT}{{\mathcal T}}
 \nc{\calU}{{\mathcal U}}
 \nc{\calV}{{\mathcal V}}
 \nc{\calW}{{\mathcal W}}
 \nc{\calX}{{\mathcal X}}
 \nc{\calY}{{\mathcal Y}}
 \nc{\calZ}{{\mathcal Z}}
  \nc{\cala}{{\mathcal a}}
 \nc{\calb}{{\mathcal b}}
 \nc{\calc}{{\mathcal c}}
 \nc{\cald}{{\mathcal d}}
 \nc{\cale}{{\mathcal e}}
 \nc{\calf}{{\mathcal f}}
 \nc{\calg}{{\mathcal g}}
 \nc{\calh}{{\mathcal h}}
 \nc{\cali}{{\mathcal i}}
 \nc{\calj}{{\mathcal j}}
 \nc{\calk}{{\mathcal k}}
 \nc{\call}{{\mathcal l}}
 \nc{\calm}{{\mathcal m}}
 \nc{\caln}{{\mathcal n}}
 \nc{\calo}{{\mathcal o}}
 \nc{\calp}{{\mathsf p}}
 \nc{\calq}{{\mathcal q}}
 \nc{\calr}{{\mathcal r}}
 \nc{\cals}{{\mathcal s}}
 \nc{\calt}{{\mathcal t}}
 \nc{\calu}{{\mathcal u}}
 \nc{\calv}{{\mathcal v}}
 \nc{\calw}{{\mathcal w}}
 \nc{\calx}{{\mathcal x}}
 \nc{\caly}{{\mathcal y}}
 \nc{\calz}{{\mathcal z}}

 \nc{\frakA}{{\mathfrak A}}
 \nc{\frakB}{{\mathfrak B}}
 \nc{\frakC}{{\mathfrak C}}
 \nc{\frakD}{{\mathfrak D}}
 \nc{\frakE}{{\mathfrak E}}
 \nc{\frakF}{{\mathfrak F}}
 \nc{\frakG}{{\mathfrak G}}
 \nc{\frakH}{{\mathfrak H}}
 \nc{\frakI}{{\mathfrak I}}
 \nc{\frakJ}{{\mathfrak J}}
 \nc{\frakK}{{\mathfrak K}}
 \nc{\frakL}{{\mathfrak L}}
 \nc{\frakM}{{\mathfrak M}}
 \nc{\frakN}{{\mathfrak N}}
 \nc{\frakO}{{\mathfrak O}}
 \nc{\frakP}{{\mathfrak P}}
 \nc{\frakQ}{{\mathfrak Q}}
 \nc{\frakR}{{\mathfrak R}}
 \nc{\frakS}{{\mathfrak S}}
 \nc{\frakT}{{\mathfrak T}}
 \nc{\frakU}{{\mathfrak U}}
 \nc{\frakV}{{\mathfrak V}}
 \nc{\frakW}{{\mathfrak W}}
 \nc{\frakX}{{\mathfrak X}}
 \nc{\frakY}{{\mathfrak Y}}
 \nc{\frakZ}{{\mathfrak Z}}
 \nc{\fraka}{{\mathfrak a}}
 \nc{\frakb}{{\mathfrak b}}
 \nc{\frakc}{{\mathfrak c}}
 \nc{\frakd}{{\mathfrak d}}
 \nc{\frake}{{\mathfrak e}}
 \nc{\frakf}{{\mathfrak f}}
 \nc{\frakg}{{\mathfrak g}}
 \nc{\frakh}{{\mathfrak h}}
 \nc{\fraki}{{\mathfrak i}}
 \nc{\frakj}{{\mathfrak j}}
 \nc{\frakk}{{\mathfrak k}}
 \nc{\frakl}{{\mathfrak l}}
 \nc{\frakm}{{\mathfrak m}}
 \nc{\frakn}{{\mathfrak n}}
 \nc{\frako}{{\mathfrak o}}
 \nc{\frakp}{{\mathfrak p}}
 \nc{\frakq}{{\mathfrak q}}
 \nc{\frakr}{{\mathfrak r}}
 \nc{\fraks}{{\mathfrak s}}
 \nc{\frakt}{{\mathfrak t}}
 \nc{\fraku}{{\mathfrak u}}
 \nc{\frakv}{{\mathfrak v}}
 \nc{\frakw}{{\mathfrak w}}
 \nc{\frakx}{{\mathfrak x}}
 \nc{\fraky}{{\mathfrak y}}
 \nc{\frakz}{{\mathfrak z}}

 \nc{\sha}{{\mbox{\cyr x}}}

\begin{document}

\title[Mordell-Tornheim Zeta  and $L$-Values]
{Reducibility of Signed Cyclic Sums of Mordell-Tornheim Zeta and $L$-Values}

\author{Jianqiang Zhao$^{\natural,\sharp}$ AND Xia Zhou$^{\dag}$}



\maketitle

\begin{center}
{${}^\natural$ Department of Mathematics, Eckerd College, St. Petersburg, FL 33711, USA\newline
${}^\sharp$ Max-Planck Institut f\"ur Mathematik, Vivatsgasse 7, 53111 Bonn, Germany\ \ \ \ \newline
${}^\dag$ Department of Mathematics, Zhejiang University, HangZhou, P.R.\ China, 310027 \ }
\end{center}

% \date{\today}
% \begin{abstract}
\vskip0.7cm
\noindent{\small {\sc Abstract.}
Matsumoto et al.\ define the Mordell-Tornheim $L$-functions of depth $k$ by
$$   L_\MT(s_1,\dots,s_{k+1};\chi_1,\dots,\chi_{k+1}) :=
    \sum_{m_1=1}^\infty \cdots \sum_{m_k=1}^\infty
   \frac{\chi_1(m_1)\cdots \chi_k(m_k)\chi_{k+1}(m_1+\cdots+m_k)}
   {m_1^{s_1}\cdots m_k^{s_k}(m_1+\cdots+m_k)^{s_{k+1}}} $$
for complex variables $s_1,\dots,s_{k+1}$ and primitive Dirichlet
characters $\chi_1,\dots,\chi_{k+1}$.
In this paper, we shall show that certain signed cyclic sums
of Mordell-Tornheim $L$-values are rational linear combinations of
products of multiple $L$-values of lower depths (i.e.,
reducible). This simultaneously generalizes some results of
Subbarao and Sitaramachandrarao, and Matsumoto et al.
As a direct corollary, we can prove that for any positive
integer $n$ and integer $k\geq 2$, the Mordell-Tornheim sums
$\zeta_\MT(\{n\}_k,n)$ is reducible where $\{n\}_k$ denotes the
string $(n,\dots,n)$ with $n$ repeating $k$ times. }
%\end{abstract}

%\tableofcontents
%\interdisplaylinepenalty=500

\vskip0.7cm

\section{Introduction}
For $\bfs=(s_1,\dots,s_{k+1})\in \C^{k+1}$ and
$\bfchi=(\chi_1,\dots,\chi_{k+1})$ where $\chi_j$'s are primitive
Dirichlet characters, Matsumoto et al. \cite{MNT}
define the \emph{Mordell-Tornheim $L$-functions} by
\begin{equation}\label{equ:MTLfunc}
L_\MT(\bfs;\bfchi):=
    \sum_{m_1=1}^\infty \cdots \sum_{m_k=1}^\infty
   \frac{\chi_1(m_1)\dots \chi_k(m_k)\chi_{k+1}(m_1+\cdots+m_k)}
    {m_1^{s_1}\cdots m_k^{s_k}(m_1+\cdots+m_k)^{s_{k+1}}}
\end{equation}
They show that when $k=3$ this function has analytic continuation
to $\C^4$. As usual we call
$|\bfs|:=s_1+\cdots+s_{k+1}$ the \emph{weight} and $k$ the \emph{depth}.
When all the characters are principle these are nothing but the
traditional Mordell-Tornheim zeta functions
\begin{equation}\label{equ:MTfuncDef}
   \zeta_\MT(s_1,\dots,s_{k+1}):= \sum_{m_1=1}^\infty \cdots \sum_{m_k=1}^\infty
   \frac{1}{m_1^{s_1}\cdots m_k^{s_k}(m_1+\cdots+m_k)^{s_{k+1}}}.
\end{equation}
Note that in the literature this function is also denoted
by $\zeta_{\MT,k}(s_1,\dots,s_k;s_{k+1})$.
One can compare \eqref{equ:MTLfunc}
to the classical multiple $L$-functions (here $\bfs=(s_1,\dots,s_k)$)
\begin{equation}\label{equ:LfuncDef}
L(\bfs;\chi_1,\dots,\chi_k) :=
    \sum_{m_1>\dots>m_k\ge 1}
    \frac{\chi_1(m_1)\dots \chi_k(m_k)}{m_1^{s_1}\cdots m_k^{s_k}}
\end{equation}
and compare \eqref{equ:MTfuncDef} to the classical multiple zeta functions
\begin{equation}\label{equ:MZfuncDef}
\zeta(\bfs) :=
    \sum_{m_1>\dots>m_k\ge 1}  \frac{1}{m_1^{s_1}\cdots m_k^{s_k}}
\end{equation}
where $|\bfs|$ is called the \emph{weight} and
$\ell(\bfs):=k$ the \emph{depth}. It is a little unfortunate, due to
historical reasons, that the ordering of the indices in (3) and (4)
is opposite to the one which naturally corresponds to that of
(1) and (2). But this ordering of multiple zeta and $L$-functions has its
advantages in many computations involving integral representations
so we choose it to be consistent with our other recent works.

In the past several decades, relations among special values of the
above functions at integers have gradually gained a lot interest among
both mathematicians and physicists. In this paper, we slightly enlarge
our scope our study. We call the number defined by~\eqref{equ:MTLfunc}
a \emph{type 1 Mordell-Tornheim $L$-value} (1-MTLV for short) if
at most one of the arguments is not a positive
integer. We call it a \emph{special 1-MTLV} if only the last variable
$s_{k+1}$ is allowed to be a complex number. We can define 1-MTZVs
and special 1-MTZVs similarly by \eqref{equ:MTfuncDef}. Parallel to these,
we can use \eqref{equ:LfuncDef} and \eqref{equ:MZfuncDef} to define
1-MLVs and 1-MZVs  (resp. their special versions) where only one variable
(resp. only the leading variable $s_1$) is allowed to be a complex number.
Note that MZVs are 1-MZVs and similarly for others. The following
diagrams provide the relations between these numbers:
\begin{equation}\label{equ:inclusions}
\begin{CD}
    \{\text{1-MTZVs} \} \ &\subset & \ \{\text{1-MTLVs} \} \\
    \cup &  \       &\ \cup       \\
    \{\text{MTZVs} \}\ &\subset & \  \{\text{MTLVs}  \}
\end{CD}\hskip2cm
\begin{CD}
    \{\text{1-MZVs} \} \ &\subset & \ \{\text{1-MLVs} \} \\
    \cup &  \       &\ \cup       \\
    \{\text{MZVs} \}\ &\subset & \  \{\text{MLVs}  \}
\end{CD}
\end{equation}

Ordinary MTZVs were first investigated by Tornheim~\cite{Torn} in the case $k=2$,
and later by Mordell~\cite{Mord} and Hoffman~\cite{Hoff} with
$s_1=\cdots=s_k=1$. On the other hand, after the seminal work of
Zagier \cite{Zag} much more results concerning MZVs have been found
(for a rather complete reference list please see Hoffman's webpage \cite{Hweb}).
Our primary interest in this paper is to study the
properties of the type 1 versions of these special values,
especially their reducibility.

\begin{defn}
A linear combination of MTZVs (resp. MTLVs, MZVs, MLVs)
is called \emph{reducible} if it can be expressed as a $\Q$-linear
combination of products of MTZVs (resp. MTLVs, MZV, MLVs)
of lower depths. It is called \emph{strongly reducible}
(not defined for MZVs and MLVs) if we can further replace
MTZVs (resp. MTLVs) by MZVs (resp. MVLs). One can similarly
define the reducibility for corresponding type 1 values.
\end{defn}
It is a well-known result that if the weight and length of a MZV have
different parities then the MZV is reducible. This was proved by Zagier
(\cite[Cor.~8]{IKZ}),
and later by Tsumura \cite{Tsu1} independently. A similar result for
MTZVs has been obtained by Bradley and the second author:
\begin{thm}    \label{thm:ZB}
Every MTZV $\zeta_\MT(\bfs)$ is a $\Q$-linearly combination
of MVZs of the same weight and length (\cite[Theorem 5]{ZB}).
Further, if $\ell(\bfs)=k+1\ge 3$ and $k+|\bfs|$ is odd then the MTZV
$\zeta_\MT(\bfs)$ is reducible and therefore strongly reducible
(\cite[Theorem 2]{ZB}).
\end{thm}
In fact, MTZVs and MZVs are closely related so it is not surprising
that similar results often hold for both. To illustrate this line
of thought, in \S\ref{equ:redResults} we shall prove
the reduction of special 1-MTLVs to special 1-MLVs, generalizing
Theorem~\ref{thm:ZB} of Bradley and the second authord. We also
present a result relating colored 1-MTZVs to colored 1-MZVs
(see Definition~\ref{defn:coloredMTZV} and \cite{BJOP}).

Like in the classical case, when the weight and depth have the same
parity the situation is more complicated. In 1985,
Subbarao and Sitaramachandrarao~\cite{SS} showed that
$\zeta_\MT(2a,2b,2c)+\zeta_\MT(2b,2c,2a)+\zeta_\MT(2c,2a,2b)$
is reducible for positive integers $a,b,c$, which includes
the special case $\zeta_\MT(2c,2c;2c)$
already known to Tornheim. In 2007, Tsumura~\cite{Tsu2} evaluated
$\zeta_\MT(a,b,s)+(-1)^b\zeta_\MT(b,s,a)+(-1)^a\zeta_\MT(s,a,b)$
for positive integers $a, b$ and complex number $s$.  Nakamura~\cite{Nak}
subsequently gave a simpler evaluation of the same quantity.
More recently, a triple 1-MTLV analog is established by
Matsumoto et al. \cite[Theorem~3.5]{MNT}
(after slight reformation): for any positive
integers $a,b,c$, and any primitive Dirichlet character $\chi$
\begin{multline}\label{equ:MatsDepth3}
(-1)^{a+b+c} L_\MT(a,b,c,s; \bfone, \bfone, \bfone,  \chi)
-(-1)^{a}L_\MT(b,c,s,a;\bfone, \bfone,  \chi, \bfone) \\
-(-1)^{b}L_\MT(c,s,a,b;\bfone,  \chi, \bfone, \bfone)
-(-1)^{c}L_\MT(s,a,b,c;\chi, \bfone, \bfone, \bfone)
\end{multline}
is reducible for all $s\in \C$ except at singular points, where
$\bfone$ is the principal character.

In this paper, we shall generalize \eqref{equ:MatsDepth3} to
arbitrary depth $k$. For a letter $v$ we denote by
$\{v\}_n$ the string with letter $v$ repeated $n$ times.
For a string $\bfw=(w_1,\dots,w_n)$,
the operator $R_j(v,\bfw)$ means to substitute
$v$ for $w_j$ if $1\le j\le n$ and $R_j(v,\bfw)=\bfw$ if $j>n$.
Using some important properties
of Bernoulli polynomials to be proved in \S\ref{sec:Bern} we shall
show in \S\ref{sec:main} the following reducibility result.
\begin{thm} \label{thm:main}
Let $k$ be a positive integer $\ge 2$ and $\bfs=(s_1,\dots,s_k)\in \N^k$. Then
\begin{equation}\label{equ:keyrel}
(-1)^{k+|\bfs|}L_\MT(\bfs,z;\{\bfone\}_k,\chi)+\sum_{j=1}^k
(-1)^{s_j}L_\MT(R_j(z,\bfs),s_j;R_j(\chi,\{\bfone\}_{k+1}))
\end{equation}
is reducible for all $z\in \C$ except at singular points.
If $z$ is also a positive integer then
\begin{equation}\label{equ:keyrel2}
\sum_{j=1}^{k+1}
(-1)^{s_j}\zeta_\MT(R_j(z,\bfs),s_j)
\end{equation}
is \emph{strongly} reducible, where $s_{k+1}=z$ .
\end{thm}
We will in fact give a precise
reduction formula in Theorem~\ref{thm:mainQuantify}
which immediately implies Theorem ~\ref{thm:main}.
Unfortunately it is too complicated to state here.
Note that \eqref{equ:keyrel} may not be strongly reducible.
We call expressions like \eqref{equ:keyrel} or \eqref{equ:keyrel2}
\emph{signed cyclic sums} of 1-MTLVs or MTZVs. The implication
$\eqref{equ:keyrel}\,\Rightarrow\, \eqref{equ:keyrel2}$
readily follows from Theorem~\ref{thm:ZB}.
Theorem \ref{thm:main} has the following nice implication.
\begin{cor} If $n\in \N$ and $k\ge 2$ then the MTZV
$\zeta_\MT(\{n\}_{k+1})$ is reducible.
\end{cor}

Note the case $n=1$ of the corollary was already treated by Mordell \cite[(5)]{Mord}:
$$\zeta_\MT(\{1\}_{k+1})=k!\zeta(k+1).$$
It would be interesting to generalize this identity to arbitrary $n$.

The main idea in the proof of Theorem~\ref{thm:main}
comes from \cite{MNT}.
Both authors would like to thank Prof. Matsumoto and Tsumura for sending them
many pre- and off-prints. The first author also wants to thank
Max-Planck-Institut f\"ur Mathematik for providing financial support
during his sabbatical leave when this work was done.
The second author is supported by the National Natural Science
Foundation of China, Project 10871169.

\section{Analytic Continuation of Mordell-Tornheim \\
Colored Zeta and $L$-functions} \label{sec:anaCont}
The main idea of this section is from \cite{Mats2,MNT} and
the result is perhaps known to the experts already.
There are three reasons we want to include this section: first,
the proof is relatively short so we can present it for completeness;
second, this is the most natural place to introduce the
term \emph{colored Mordell-Tornheim functions}
to be used later in the paper;
and last, our main results Theorem~\ref{thm:main} and
Theorem~\ref{thm:mainQuantify} rely on
the analytic continuation of Mordell-Tornheim
colored zeta and $L$-functions.

Clearly when $\Re(s_j)>1$ all the functions in (1) to (4) converge.
Denote the real part of $s_j$ by $\Re(s_j)=\sigma_j$
for $1\le j\le k+1$ and write $s=s_{k+1}$ and $\gs=\gs_{k+1}$.
Since \eqref{equ:MTfuncDef} (resp. \eqref{equ:MTLfunc}) remains unchanged if the
arguments $s_1,\dots,s_k$ (resp. $(s_1,\chi_1),\dots,(s_k,\chi_k)$)
are permuted,  we may as well suppose that $s_1,\dots,s_k$ are
arranged in the order of increasing real parts,
i.e., $\sigma_1\le\cdots\le \sigma_k$. It follows from
\cite[Theorem~4]{ZB} that the series~\eqref{equ:MTfuncDef}
converges absolutely if
\begin{equation}\label{equ:convDomain}
 \sigma+\sum_{j=1}^r \sigma_j > r, \quad \forall r=1,2,\dots,k.
\end{equation}

However, just like the Riemann zeta functions, all the functions in
\eqref{equ:MTLfunc} to \eqref{equ:MZfuncDef}
should have analytic continuations to the whole complex space with clearly
described singularities lying inside at most countably many hyperplanes.
For multiple zeta and $L$-functions this has been worked out
in \cite{AI} and \cite{Zana} (independently \cite{AET}), respectively.
Further, Matsumoto et al.\ have studied the Mordell-Tornheim zeta functions
completely (see \cite[Theorem~6.1]{MNOT}). On the other hand, in \cite{MNT}
Matsumoto et al.\ only treated depth three Mordell-Tornheim $L$-functions
although it is possible to combine their ideas in \cite{MNT} and
\cite[Theorem~6.1]{MNOT} to prove the general cases which we shall
carry out in Theorem~\ref{thm:anaCont}. To prepare for this we first
introduce the ``colored'' version of the Mordell-Tornheim functions.
In the depth 1 case, this is a
special case of the more general Lerch series. In depth 2, Nakamura called
these functions ``double Lerch serires'' (see \cite{Nak2}).

\begin{defn}\label{defn:coloredMTZV}
For any $x\in \R$ set $e(x)=e^{2\pi i x}$. For any given set of
parameters $\bfga=(\ga_1,\cdots,\ga_{k+1})\in \R^{k+1}$,
we define the function in complex variables $\bfs\in \C^{k+1}$
\begin{equation}\label{equ:claim}
 \zeta_\MT(\bfs;\bfga)= \sum_{m_1=1}^\infty \cdots \sum_{m_k=1}^\infty
   \frac{ e(\ga_1 m_1+\cdots+\ga_k m_k+\ga_{k+1}(m_1+\cdots+m_k))}
    {m_1^{s_1}\cdots m_k^{s_k}(m_1+\cdots+m_k)^{s_{k+1}}},
\end{equation}
where $\Re(s_j)\ge 1$ for all $j\le k+1$.
This is called a \emph{colored Mordell-Tornheim function} with
variables $s_j$ dressed with $e(\ga_j)$.
\end{defn}
This terminology is influenced by the name ``colored MZVs'' used in
\cite{BJOP} in which similar generalizations of MZVs are considered.
Of course colored MZVs can also be regarded as special values of
multiple polylogarithms on the unit circle. To study $L$-functions
we only need the ``colors'' to be roots of unity (i.e. $\ga_j\in \Q$)
in which case the colored MZVs have been investigated from
different points of view in \cite{DG,Rac,Znote,Zpolrel}.

For any $\bfgb=(\gb_1,\dots,\gb_r)\in \R^r$ and a set $A$ we let
$A(\bfgb)=A$ if $\bfgb\in \Z^r$ and $A(\bfgb)=\emptyset$ otherwise.
Similar to the proof of \cite[Prop.~2.1]{MNT} we first have:
\begin{prop}\label{prop:colMT}
Let $\gb_j=\ga_j+\ga_{k+1}$ for all $j=1,\dots,k$.
Then the colored Mordell-Tornheim functions defined
by \eqref{equ:claim} can be analytically continued to $\C^{k+1}$
with the singularities lying on the following hyperplanes:
\begin{equation}\label{equ:poleSet}
\{|\bfs|=k\}\cup \bigcup_{l=1}^\infty
    \bigcup_{r=1}^{k-1} \bigcup_{1\le j_1<\cdots<j_r\le k}
    \left\{ \sum_{i=1}^r (s_{j_i}-1)+s_{k+1}=-l\right\}(\gb_{j_1},\dots,\gb_{j_r}).
\end{equation}

When all $\gb_j=0$ we recover the analytic continuation
of Mordell-Tornheim functions defined by \eqref{equ:MTfuncDef}
(cf.\ \cite[Theorem~6.1]{MNOT}).
\end{prop}
\begin{proof}
We proceed by induction on the depth. The case of depth two is given by
\cite[Theorem~1]{Mats}. Assume in depth $k-1$ ($k\ge 3$) we
already have the analytic continuations of $\zeta_\MT(\bfs;\bfga)$
for every fixed $\bfga=(\ga_1,\cdots,\ga_k)$ with the singularities
given by the proposition. Let's consider the depth $k$ situation in
\eqref{equ:claim}. First we assume $\Re(s_j)\ge 1$ for all $j=1,\dots,k+1$.
By substitution $\ga_j\to \ga_j-\ga_{k+1}$ for all $j=1,\dots,k$
we may also assume without loss of generality that $\ga_{k+1}=0$.
Set $n_r=\sum_{j=1}^r m_j$ for $r=1,\dots,k+1$.
By Mellin-Barnes formula for all $b>0$ we have
\begin{equation}\label{equ:MB}
    \frac1{(1+b)^s}=\frac1{2\pi i}\int_{\text{($c$)}}
    \frac{\gG(s+z)\gG(-z)}{\gG(s)}b^z\,dz
\end{equation}
where $s\in\C$, $\Re(s)>-c>0$ and  ($c$) is the vertical line
$\Re(z)=c$ pointing upward. Applying this with
$b=m_{k+1}/n_k$ we get
$$\frac1{n_{k+1}^{s_{k+1}}}=\frac1{n_k^{s_{k+1}}}\frac1{2\pi i}
    \int_{\text{($c$)}} \frac{\gG(s_{k+1}+z)\gG(-z)}{\gG(s_{k+1})}
    \left(\frac{m_{k+1}}{{n_k}}\right)^z\,dz $$
where $\Re(s_{k+1})>-c>0$. Setting $\bfga=(\ga_1,\dots,\ga_k,0)$ we see that
\begin{multline}\label{equ:IndictionRel}
   \zeta_\MT(\bfs;\bfga)= \frac1{2\pi i}
    \int_{\text{($c$)}} \frac{\gG(s_{k+1}+z)\gG(-z)}{\gG(s_{k+1})}
    \zeta_\MT(\bfs',s_{k+1}+z;\bfga',0) \phi(s_k-z,\ga_k)\,dz
\end{multline}
where $\bfs'=(s_1,\dots,s_{k-1})$, $\bfga'=(\ga_1,\dots,\ga_{k-1})$, and
$\phi(s,\ga)=\sum_{j\ge 1} e(j\ga)/j^s$. Note that
$\zeta_\MT(\bfs',s_{k+1}+z;\bfga',0)$ is well defined by our
assumption $\Re(s_j)\ge 1$, $\Re(s_{k+1}+z)=\Re(s_{k+1})+c>0$
and \eqref{equ:convDomain}. As $|e(\ga m)|=1$
we still have the exponential decay of the integrand in \eqref{equ:IndictionRel}
by Stirling's formula when $z\to c\pm i\infty$ and therefore the argument for
\cite[(3.2)]{Mats2} carries through without problem.
When we shift the integration from ($c$) to ($M-\vep)$
for large $M\in \N$ and very small $\vep>0$ we need to consider
the residues of the integrand of \eqref{equ:IndictionRel}
between these two vertical lines.
By the induction assumption, the singularities of
$\zeta_\MT(\bfs',s_k+z;\bfga',0)$ are given by:
\begin{multline*}
  \left\{\sum_{j=1}^{k-1} s_j+s_{k+1}+z =k-1\right\}\cup \\
    \cup \bigcup_{l=1}^\infty\bigcup_{r=1}^{k-2}
    \bigcup_{1\le j_1<\cdots<j_r\le k-1}
    \left\{ \sum_{i=1}^r (s_{j_i}-1) +s_{k+1}+z=-l\right\}(\ga_{j_1},\dots,\ga_{j_r}).
\end{multline*}

By assumption $\Re(s_j)\ge 1$ ($j\le k$) and $\Re(s_{k+1}+z)>0$
none of these lies between the two vertical lines
so the only relevant poles of the integrand of \eqref{equ:IndictionRel}
are $z=j$ for $j=0,1,\dots,M-1$ given by $\gG(-z)$
and $\{z=s_k-1\}(\ga_k)$ given by $\phi(s_k-z,\ga_k)$.
It is well-known that these poles are all simple poles.
Thus by contour integration we get:
\begin{align*}
  \zeta_\MT(\bfs;\bfga)=& \frac1{2\pi i}
    \int_{\text{($M-\vep$)}} \frac{\gG(s_{k+1}+z)\gG(-z)}{\gG(s_{k+1})}
    \zeta_\MT(\bfs',s_{k+1}+z;\bfga',0) \phi(s_k-z,\ga_k)\,dz\\
 -&\left[\frac{\gG(s_{k+1}+s_k-1)\gG(1-s_k)}{\gG(s_{k+1})}
    \zeta_\MT(\bfs',s_{k+1}+s_k-1;\bfga',0)\right]_{\displaystyle \ga_k} \\
 +&\sum_{j=0}^{M-1}{-s_{k+1}\choose j}
    \zeta_\MT(\bfs',s_{k+1}+j;\bfga',0)\phi(s_k-j,\ga_k)
 \end{align*}
where $[x]_\ga=x$ if $\ga\in\Z$ and $[x]_\ga=0$ otherwise.
As $M$ can be arbitrarily large a careful computation
now yields the correct set of poles for $\zeta_\MT(\bfs;\bfga)$
as given in \eqref{equ:poleSet}. This concludes the proof of the proposition.
\end{proof}

\begin{thm}\label{thm:anaCont}
The Mordell-Tornheim $L$-function $L_\MT(\bfs;\bfchi)$ defined by \eqref{equ:MTLfunc}
can be analytically continued to a memorphic function over $\C^{|\bfs|}$
with  explicitly computable singularities lying in at most
countably many hyperplanes.
\end{thm}
\begin{proof}
Let $f_j$ be the conductor of $\chi_j$ for $j=1,\dots,k+1$.
By \cite[Lemma~4.7]{Wash} we see that
\begin{equation}\label{equ:LtoMT}
 L_\MT(\bfs;\bfchi)=
    \sum_{j_1=1}^{f_1}\cdots\sum_{j_{k+1}=1}^{f_{k+1}}
    \prod_{i=1}^{k+1} \frac{\chi_i(j_i)}{\tau(\ol{\chi_i})}
  \zeta_\MT(\bfs;j_1/f_1,\dots, j_{k+1}/f_{k+1})
\end{equation}
where $\tau(\ol{\chi_i})$ is the Gauss sum. Hence the theorem
follows from Prop.~\ref{prop:colMT} immediately.
\end{proof}

\section{Reducing Mordell-Tornheim Type Values to Traditional Values}\label{equ:redResults}
In this section we shall prove that the study of special 1-MTLVs and
colored special 1-MTZVs can be reduced to that of special 1-MLVs
and special 1-MZVs, respectively.

\begin{thm} \label{thm:MTLV=MLV}
Fix a positive integer $k\ge 2$. Let $z\in \C$
and $\bfs=(s_1,\dots,s_k)\in \N^k$. Then for any primitive
Dirichlet character $\chi$ the special 1-MTLV
$L_\MT(\bfs,z; \{\bfone\}_k,\chi)$ is a $\Q$-linear combination
of special 1-MLVs of the same weight and depth
and of the same character type $(\{\bfone\}_k,\chi).$
\end{thm}
\begin{proof} Essentially the same proof of \cite[Theorem~5]{ZB}
works here. For example, with their notation we can multiply
$\chi(n_r)$ inside each sum appearing
in their proof. Moreover, the variable $s$ always appears in
the last variable position of every function $T_\ell$ throughout
the proof. This corresponds to the leading position of the 1-MLV
so the values are always special.
\end{proof}

Due to the combinatorial nature of the proof it won't work
for non-special 1-MTLVs or wrong character types.
In fact, more generally, every special 1-MTLV
$L_\MT(\bfs;\chi_1,\dots,\chi_{k+1})$
is a $\Q$-linear combination of the following values
$$  \sum_{m_1=1}^\infty \cdots \sum_{m_k=1}^\infty
   \frac{\chi_1(m_1)\chi_2(m_2)\cdots \chi_k(m_k)\chi_{k+1}(m_1+\cdots+m_k)}
   {m_1^{r_1}(m_1+m_2)^{r_2}\cdots (m_1+\cdots+m_{k-1})^{r_{k-1}}
   (m_1+\cdots+m_k)^{s_{k+1}+r_k}} $$
where $\bfr\in \Z^k$. Notice that these are not 1-MLVs as
Dirichlet characters are not additive in general.
The situation for colored 1-MTZVs is little better, with no
restriction on the type of the ``colors''.

\begin{thm} \label{thm:coloredMTZV=MZVs}
Every colored special 1-MTZVs is a $\Q$-linear combination of colored
special 1-MZVs of the same weight and same depth. More precisely,
the colored special 1-MTZV $\zeta_\MT(s_1,\dots,s_{k+1}; \ga_1,\dots,\ga_{k+1})$
is a $\Q$-linear combination of colored special 1-MZVs of the following form
$$ \zeta(s_{k+1}+r_k, r_{k-1},\dots, r_1;
 \ga_k+\ga_{k+1},\ga_{k-1}-\ga_k,\dots,\ga_1-\ga_2),$$
where $\bfr\in \Z^k$.
\end{thm}
\begin{proof} Modify the proof of \cite[Theorem~5]{ZB}
by inserting ``colors'' into expressions of $T_\ell$'s there.
\end{proof}

\begin{rem}
Theorem \ref{thm:MTLV=MLV} and Theorem \ref{thm:coloredMTZV=MZVs}
generalize \cite[Theorem~5]{ZB} about MTZVs in two different
directions: the former to their $L$-function version while
the latter to their colored version.
\end{rem}

\section{Preliminaries on Bernoulli Polynomials}\label{sec:Bern}
By definition the Bernoulli polynomials $B_n(x)$
are periodic functions with period 1
defined by the generating function:
$$\frac{te^{xt}}{e^t-1}=\sum_{n\ge 0} B_n(x)\frac{t^n}{n!},\qquad x\in[0,1).$$
The values $B_n:=B_n(0)$ are the Bernoulli numbers which are linked to the
Riemann zeta values :
\begin{equation}\label{equ:Bernzeta}
B_1=\zeta(0)=-\frac12, \quad
B_{2s}=-\frac{2(2s)!}{(2\pi i)^{2s} }\zeta(2s)\quad \forall s\in \N.
\end{equation}

For any positive integer $s$ and $N$ we set
$$f^{+}_{s,N}(x)=\sum_{k=1}^N \frac{e(kx)}{k^s},\qquad
f^{-}_{s,N}(x)=\sum_{k=-1}^{-N} \frac{e(kx)}{k^s}=(-1)^s f^{+}_{s,N}(-x),$$
where $e(x)=e^{2\pi i x}$, and
$$ f_{s,N}(x)=f^{+}_{s,N}(x)+f^{-}_{s,N}(x).$$
\begin{lem} \label{lem:Bernpoly}
For every positive integer $s$ we have
\begin{equation}\label{equ:BernSeries}
B_s(x)=-\frac{s!}{(2\pi i)^s}  f_{s, \infty}(x).
\end{equation}
Its derivative
\begin{equation}\label{equ:Bernder}
B'_s(x)=s B_{s-1}(x)
\end{equation}
and for $m\ne 0$ the integral
\begin{equation}\label{equ:BernInt}
 \int_0^1 B_n(x) e(m x)\, dx=-\gam_{0,n}\frac{n!}{(-2\pi i m)^n}
\end{equation}
where $\gam_{0,n}=0$ if $n=0$ and $\gam_{0,n}=1$ otherwise.
For $\bfs=(s_1,\dots,s_t)\in \N^t$
\begin{equation}\label{equ:BernProdConst}
C_\bfs:=\int_0^1 \prod_{j=1}^t B_{s_j}(x) \, dx=
\sum_{r_1=0}^{s_1}\cdots \sum_{r_t=0}^{s_t} {s_1\choose r_1}\cdots {s_t\choose r_t} \frac{B_{\bfs-\bfr}}{|\bfr|+1},
\end{equation}
where for any vector $\bfv=(v_1,\dots,v_t)$ we set $|\bfv|:=\sum_{j=1}^t v_j$
and $B_\bfv:=\prod_{j=1}^t B_j$.
\end{lem}
\begin{proof} All the statements are well-known (for e.g., see  \cite[pp.~804--805]{AS})
except perhaps \eqref{equ:BernSeries} which is on \cite[p.362]{Carlitz}.
\end{proof}

Let $[t]:=\{1,\dots,t\}$ with the increasing order. By abuse of notation we let
$\bfi=(i_1,\dots,i_\gl)\subseteq [k]$ denote both a set and a vector
such that $i_1<\dots<i_\gl$. No confusion should arise. We denote
its length by $\ell(\bfi)=\gl$ and write $\bfi!=i_1!\cdots i_\gl!$.
If $\ell(\bfj)=\ell(\bfi)< t$ then we define the
inflation of the vector $\bfj$ to length $t$ with respect to $\bfi$ as
\begin{equation*}
 \Inf^t_\bfi(\bfj)=(l_1,\dots,l_t), l_\gb=
\begin{cases}
j_\ga \quad & \text{if }\gb=i_\ga\in \bfi;\\
1   &\text{if }\gb\not\in \bfi.
\end{cases}
\end{equation*}
Note that $|\Inf^t_\bfi(\bfj)|=|\bfj|-\ell(\bfi)+t$.
This operation essentially stretches the vector $\bfj$ to a length $t$ vector
by redistributing its entries to $\bfi$-th positions while inserting
$1$'s in other positions. Finally, for a vector $\bfv=(v_1,\dots,v_t)$
we write
$${|\bfv|\choose \bfv}={|\bfv|\choose {v_1,\dots,v_t}}.$$
The next proposition is not needed in the proof of the main results
in the paper but it offers a simple and close expression
of an arbitrary product of Bernoulli polynomials and therefore
should have independent interest by itself. It generalizes the
well-known result of Calitz \cite{Carlitz}.

\begin{prop}\label{prop:BerProd}
Keep the same notation as in Lemma \ref{lem:Bernpoly}.  Then
\begin{equation}\label{equ:BerProd}
B_\bfs(x)=C_\bfs+\sum_{ \bfi\subsetneq [t] }
 \sum_{0\le j_\bfi\le s_\bfi}
 {|\bfs|-|\bfj|+\ell(\bfi)-t\choose  \bfs -\Inf^t_\bfi(\bfj)} \frac{B_\bfj}{\bfj!}\cdot
\frac{\bfs! B_{|\bfs|-|\bfj|+\ell(\bfi)-t+1}(x)}{(|\bfs|-|\bfj|+\ell(\bfi)-t+1)!}.
\end{equation}
where for an vector $\bfi=(i_1,\dots,i_\ell)$ we write the multiple sum
$\displaystyle \sum_{0\le j_\bfi\le s_\bfi} =
\sum_{j_1=0}^{s_{i_1}}\cdots \sum_{j_\ell=0}^{s_{i_\ell}}.$
\end{prop}
\begin{proof} By induction it is easy to show that
\begin{equation}\label{equ:parSum}
 \prod_{\gt=1}^t \frac{1}{e^{u_\gt}-1}=\frac{1}{e^{|\bfu|}-1}
\sum_{\bfi\subsetneq [t]} \prod_{\gt=1}^{\ell(\bfi)} \frac{1}{e^{u_{i_\gt}}-1},
\end{equation}
where the product on the right is 1 when $\bfi=\emptyset$. Indeed,
if $t=1$ then $\bfi$ has to be $\emptyset$ and \eqref{equ:parSum} is clear.
Assume  \eqref{equ:parSum} holds for $t\ge 1$. Then we have
\begin{align*}
  \ & \frac{1}{e^{u_{t+1}}-1}\prod_{\gt=1}^t \frac{1}{e^{u_\gt}-1}=
\frac{1}{(e^{u_{t+1}}-1)(e^{|\bfu|}-1)}
\sum_{\bfi\subsetneq [t]} \prod_{i\in\bfi} \frac{1}{e^{u_i}-1}\\
=& \frac{1}{e^{|\bfu|+u_{t+1}}-1}
    \left(1+\frac{1}{e^{u_{t+1}}-1}+\frac{1}{e^{|\bfu|}-1}\right)
    \left(1+\sum_{\emptyset\ne\bfi\subsetneq [t]}
    \prod_{i\in\bfi} \frac{1}{e^{u_i}-1}\right)\\
=&\frac{1}{e^{|\bfu|+u_{t+1}}-1}\prod_{\gt=1}^{t+1}
    \left(1+\sum_{\emptyset\ne\bfi\subsetneq [t+1]}
    \prod_{i\in\bfi} \frac{1}{e^{u_i}-1}\right).
\end{align*}
Thus \eqref{equ:parSum} is proved. Applying it we
may transform the following power series ($\bfu=(u_1,\dots,u_t)$)
\begin{align*}
   |\bfu| \sum_{\bfs\in (\Z_{\ge 0})^t}
    \frac{B_\bfs(x)}{\bfs!} u_1^{s_1}\cdots u_t^{s_t}=&\frac{|\bfu|
    \prod_{\gt=1}^t u_\gt}{\prod_{\gt=1}^t(e^{u_\gt}-1)}e^{x|\bfu|}
    =\frac{|\bfu|e^{x|\bfu|}}{e^{|\bfu|}-1}
   \sum_{\bfi\subsetneq [t]} \prod_{i\in\bfi}
    \frac{u_i} {e^{u_i}-1} \prod_{i\not\in \bfi} u_i \\
   =&\sum_{n=0}^\infty B_n(x)\frac{|\bfu|^n}{n!}\sum_{\bfi\subsetneq [t]}
   \prod_{i\not\in \bfi} u_i \
   \sum_{0\le j_\bfi< {\boldsymbol \infty}}
   \frac{B_\bfj}{\bfj!}\prod_{\gt=1}^{\ell(\bfi)} u_{i_\gt}^{j_\gt}
\end{align*}
where ${\boldsymbol \infty}=(\infty,\dots,\infty)$.
On the left, the coefficient for $u_1^{s_1}\cdots u_t^{s_t}/\bfs!$ is
$$\sum_{\gt=1}^t  s_\gt B_{s_1}(x)\cdots B_{s_{\gt-1}}(x)
    B_{s_\gt-1}(x)B_{s_{\gt+1}}(x)\cdots B_{s_t}(x)=(B_\bfs(x))'$$
by \eqref{equ:Bernder}. Integrating this we get
$$B_\bfs(x)=C+\sum_{ \bfi\subsetneq [t] }
\sum_{0\le j_\bfi\le s_\bfi}
 {|\bfs|-|\bfj|+\ell(\bfi)-t\choose  \bfs -\Inf^t_\bfi(\bfj)} \frac{B_\bfj}{\bfj!}\cdot
\frac{\bfs! B_{|\bfs|-|\bfj|+\ell(\bfi)-t+1}(x)}{(|\bfs|-|\bfj|+\ell(\bfi)-t+1)!},$$
for some constant $C$. Integrating again and noticing that $\int_0^1 B_n(x)\,dx=0$
whenever $n\ge 1$ we get $C=C_\bfs$ by \eqref{equ:BernProdConst}. This completes
the proof of the lemma.
\end{proof}

The above lemma expresses the product of different Bernoulli polynomials explicitly
as a linear combination of Bernoulli polynomials of different degrees. But it
has the drawback that we cannot restrict the degrees in order to provide
a general reduction formula in Theorem~\ref{thm:main} since the
terms on the right hand side of \eqref{equ:BerProd} do not always have the
same weight where the weight of a product term in \eqref{equ:BerProd}
is the sum of the indices of the
all the Bernoulli numbers appearing in that product
(this comes from relation \eqref{equ:Bernzeta}). However,
Prop.~\ref{prop:BernProdNice} will enable us to quantify Theorem~\ref{thm:main}
even though it has much more complex structure than Prop.~\ref{prop:BerProd}.
To state it we need some more definitions and notations.

\begin{defn} \label{defn:prefatPartS}
For arbitrary $\bfs=(s_1,\dots,s_t)\in \N^t$ a \emph{partition} of $\bfs$
is always an ordered partition $\bfP=(\bfP_1,\dots,\bfP_q)$ such that
the concatenation of $\bfP$ is $\bfs$. A
\emph{pre-fat partition} of $\bfs$ is such a partition
with $l_j:=\ell(\bfP_j)\ge 2$ for all $j\le q-1$. Its
\emph{pre-associated index set}
is the set $\ind'(\bfP)$ of indices $\bfr=(\bfr_1,\dots,\bfr_q)$ where
\begin{equation*}
  \bfr_j=\left\{
   \begin{array}{ll}
      (r_{j,1},r_{j,2},\dots,r_{j,l_j-2} ) & \hbox{if $j\le q-1$;} \\
      (r_{j,1},r_{j,2},\dots,r_{j,l_j-1} ) & \hbox{if $j=q$ (vacuous if $l_q=1$).}
   \end{array}
 \right.
\end{equation*}
Setting $\bfP_j:=(s_{j,1},s_{j,2},\dots,s_{j,l_j})$ for $j=1,\dots,q$.
For each $j$ and each $i=1,2,\dots, \ell(\bfr_j)$, the range of the
integer index $r_{j,i}$ goes from $0$ to
$\lfloor\max\{\gs_i(\bfP_j)-2\gs_{i-1}(\bfr_j), s_{j,i+1}\}/2\rfloor$
where for any vector $\bfv=(v_1,\dots,v_\ell)$ we denote its $i$-th partial
sum by $\gs_i(\bfv):=v_1+\dots+v_i$ and $\gs_0(\bfv):=0$.
\end{defn}

\begin{defn} \label{equ:fatPartS}
A \emph{fat partition} of $\bfs$ is a pre-fat partition
$\bfP=(\bfP_1,\dots,\bfP_q)$ such that
$l_q\ge 2$, i.e., every part has length at least two.
Its \emph{associated index set} is the set $\ind(\bfP)$ of
$\bfr=(\bfr_1,\dots,\bfr_q)$ where for each $1\le j\le q$,
\begin{equation*}
  \bfr_j= (r_{j,1},r_{j,2},\dots,r_{j,l_j-2} )
\end{equation*}
with each $r_{j,i}$ running over the same range as above in
Definition~\ref{defn:prefatPartS}.
\end{defn}
It is an easy exercise to see that the number of fat partitions
of $\bfs=(s_1,\dots,s_t)$ is given by the Fibonacci number
$F_{t-1}$ \cite[p.~46, \textbf{14.b}]{Stan}, where
$F_1=F_2=1$ and $F_{n+2}=F_{n+1}+F_n$ for all $n\ge 1$. Obviously
the number of pre-fat partitions of $\bfs$ is given by $F_t$.

\begin{prop} \label{prop:BernProdNice}
Let $\calP'(\bfs)$ (resp.\ $\calP(\bfs)$) be the set of pre-fat (resp.\ fat)
partitions of $\bfs=(s_1,\dots,s_t)\in \N^t$ with $t\ge 2$. For each partition
$\bfP$ let $q:=q(\bfP)$ be the number of parts in $\bfP$. Then
\begin{equation} \label{equ:BernProdNice}
\aligned
  B_\bfs(x)= &\sum_{\bfP\in \calP'(\bfs)} \sum_{\bfr\in\ind'(\bfP)}
   \prod_{j=1}^q \Bigg\{\Bigg\{\prod_{i=1}^{\ell(\bfr_j)}
   b_{j,i}(\bfP,\bfr)\Bigg\} B_{j}(\bfP,\bfr,x)\Bigg\}  \\
+& \sum_{\bfP\in \calP(\bfs)} \sum_{\bfr\in\ind(\bfP)}
   \prod_{j=1}^q \Bigg\{\Bigg\{ \prod_{i=1}^{\ell(\bfr_j)}
   b_{j,i}(\bfP,\bfr)\Bigg\}B_{j}(\bfP,\bfr)\Bigg\},
\endaligned
\end{equation}
where
\begin{equation}\label{equ:B's}
  B_{j}(\bfP,\bfr)= (-1)^{1+s_{j,l_j}}\frac{(|\bfP_j|-s_{j,l_j}-2|\bfr_j|)!(s_{j,l_j}) !}
          {(|\bfP_j|-2|\bfr_j|)!}B_{|\bfP_j|-2|\bfr_j|},
\end{equation}
where $s_{j,l_j}$ is the last component of $\bfP_j$, and
\begin{equation}\label{equ:Bx's}
  B_{j}(\bfP,\bfr,x)=\left\{
     \begin{array}{ll}
       B_{j}(\bfP,\bfr) , & \hbox{if $j<q$;} \\
       {\displaystyle  B_{|\bfP_q|}(x),\phantom{ \frac{A}{B}}}& \hbox{if $j=q$ and $l_q=1$;} \\
      {\displaystyle B_{|\bfP_q|-2|\bfr_q|}(x) \phantom{ \frac{A}{B}},} & \hbox{if $j=q$ and $l_q>1$.}
     \end{array} \right.
\end{equation}
By convention, if $l(\bfr_j)=0$ then the innermost product is $1$.
If $l(\bfr_j)\ge 1$ then $b_{j,i}(\bfP,\bfr)=$
\begin{equation}\label{equ:b's}
 \bigg[{\gs_i(\bfP_j)-2\gs_{i-1}(\bfr_j)\choose 2r_{j,i}} s_{j,i+1}
       +{s_{j,i+1}\choose
            2r_{j,i}}\big(\gs_i(\bfP_j)-2\gs_{i-1}(\bfr_j)\big)\bigg] \frac{B_{2 r_{j,i}}}{\gs_{i+1}(\bfP_j)-2\gs_i(\bfr_j)}.
\end{equation}
\end{prop}
\begin{proof}
We prove the proposition by induction on $t$. If $t=2$ then
the proposition has a very explicit form given by \cite[(3)]{Carlitz}
or \cite{Nielsen}:
 \begin{equation}\label{equ:CarLg2}
  B_{s_1}(x)B_{s_2}(x)=\sum_{r=0}^{\lfloor \max\{s_1,s_2\}/2\rfloor} \left[
{s_1 \choose 2r}s_2 +{s_2 \choose 2r}s_1 \right]\frac{B_{2r}B_{|\bfs|-2r}(x)}{|\bfs|-2r}
-(-1)^{s_2}\frac{s_1!  s_2!}{|\bfs| !} B_{|\bfs|},
 \end{equation}
where $|\bfs|=s_1+s_2$.
Let's check formula \eqref{equ:BernProdNice} is correct. In this case
the only pre-fat partition is the whole $\bfs=\bfP$ because in every pre-fat
partition only the last part can have length equal to one. So $q=1$,
$\ind'(\bfP)=\{r: 1\le r\le \lfloor\max\{s_1,s_2\}/2\rfloor\}$
and $\ind(\bfP)=\emptyset$. Then
$$B_{s_1}(x)B_{s_2}(x)=\sum_{r\in\ind'(\bfP)}
 b_{1,1}(\bfP,r)  B_{1}(\bfP,r,x)
+ B_{1}(\bfP,\emptyset)
$$
where $B_{1}(\bfP,r,x)=B_{|\bfs|- 2r}(x),$ $B_{1}(\bfP,\emptyset)=
     -(-1)^{s_2}s_1!s_2! B_{|\bfs|}/(|\bfs|) !$ and
$$b_{1,1}(\bfP,r) =\bigg[{s_1\choose 2r} s_2
       +{s_2\choose  2r }s_1 \bigg]  \frac{B_{2r}}{|\bfs|- 2r} .$$
Thus the case $t=2$ is verified.

Assume the proposition is true
when $\ell(\bfs)=t\ge 2$. Then we can use \eqref{equ:BernProdNice} to
compute $B_\bfs(x)B_n(x)$ for any positive integer $n$.
Clearly when $B_n(x)$ is multiplied by the sums involving only
Bernoulli numbers (the second line of \eqref{equ:BernProdNice})
we get exactly those terms corresponding to the pre-fat
partitions $\bfQ$ of $(s_1,\dots,s_t,n)$ whose last part has length one.
This can be readily explained by the map
\begin{align}
  \calP(\bfs)  \longrightarrow & \calP'((\bfs,n))  \notag\\
  \bfP \longmapsto &  \bfQ':=(\bfP,(n)) . \label{equ:map1}
\end{align}
It is obvious that the pre-associated index set of $\bfQ'$
is exactly the same as the associated index set of $\bfP$.
When $B_n(x)$ is multiplied on each of the terms in the first
nested sum of \eqref{equ:BernProdNice}
two kind of terms will appear according to \eqref{equ:CarLg2}.
Let's consider the following two cases: (i) $l_q=1$, and (ii)  $l_q>1$.

In case (i) we have
\begin{multline*}
 B_{|\bfP_q|}(x)B_n(x)=\sum_{r=0}^{\lfloor \max\{|\bfP_q|,n\}/2\rfloor} \left[
{|\bfP_q| \choose 2r}n+{n \choose 2r}|\bfP_q|\right]
\frac{B_{2r}B_{|\bfP_q|+n-2r}(x)}{|\bfP_q|+n-2r} \\
+(-1)^{1+n}\frac{|\bfP_q|!n!}{(|\bfP_q|+n)!} B_{|\bfP_q|+n}.
\end{multline*}
The summation term in the above contribute exactly to those pre-fat
partitions $\bfQ'$ of $(s_1,\dots,s_t,n)$ whose last part has
length equal to two. The last term of the above corresponds
to the fat partitions $\bfQ$ of $(s_1,\dots,s_t,n)$
whose last part has length equal to two. This can be summarized by the map
\begin{align}
  \calP'(\bfs)  \longrightarrow & \calP'((\bfs,n))\times \calP((\bfs,n))  \notag\\
  \bfP \longmapsto &  \bfQ'=:\big(\bfP_1,\dots,\bfP_{q-1}, (\bfP_q,n)\big),
         \bfQ:=\big(\bfP_1,\dots,\bfP_{q-1}, (\bfP_q,n)\big) . \label{equ:map2}
\end{align}
It's easy to check that the pre-associated index set of $\bfQ'$ is obtained
from the pre-associated index set $\bfr$ of $\bfP$ by adding one more part
at the end: $(r)$ itself alone, which goes from $0$ to
$\lfloor \max\{|\bfP_q|,n\}/2\rfloor$.
The corresponding term in \eqref{equ:BernProdNice} is thus determined by
\eqref{equ:b's} and the third case of \eqref{equ:Bx's}. It's
also clear that the associated index set of $\bfQ$ is equal to $\bfr$ which
is consistent with \eqref{equ:B's} since the last component of $\bfQ$ has
length 2 which implies that the innermost product of the second sum
in \eqref{equ:BernProdNice} is 1 by convention.

In case (ii) we have
\begin{multline*}
 B_{|\bfP_q|-2|\bfr_q|}(x)B_n(x)=\sum_{r=0}^{\lfloor \max\{|\bfP_q|-2|\bfr_q|,n\}/2\rfloor} \left[
{|\bfP_q|-2|\bfr_q| \choose 2r}n+{n \choose 2r}(|\bfP_q|-2|\bfr_q|)\right]\\
\cdot \frac{B_{2r}B_{|\bfP_q|+n-2|\bfr_q|-2r}(x)}{|\bfP_q|+n-2|\bfr_q|-2r}
+(-1)^{1+n}\frac{(|\bfP_q|-2|\bfr_q|)!n!}{(|\bfP_q|+n-2|\bfr_q|) !} B_{|\bfP_q|+n-2|\bfr_q|}.
\end{multline*}
Similarly to the above, this can be summarized by the map
\begin{align}
  \calP'(\bfs)  \longrightarrow & \calP'((\bfs,n))\times \calP((\bfs,n))  \notag\\
  \bfP \longmapsto &  \bfQ'=:\big(\bfP_1,\dots,\bfP_{q-1}, (\bfP_q,n)\big),
         \bfQ=:\big(\bfP_1,\dots,\bfP_{q-1}, (\bfP_q,n)\big)   \label{equ:map3}
\end{align}
with the last component in both $\bfQ'$ and $\bfQ$ having
length greater than 2. It is easy to check that
index sets is consistent with \eqref{equ:BernProdNice} when $\bfs$ is replace
by $(\bfs,n)$ (let's call the equation after such a change \eqref{equ:BernProdNice}$'$)
by inserting $r$ into the end of the last component of $\bfr$.

The above argument shows that every term in the product expansion of
$B_{|\bfs|}(x)B_n(x)$ appears in \eqref{equ:BernProdNice}$'$.
Finally, one can check that in \eqref{equ:BernProdNice}$'$ every term
is produced exactly once by the maps \eqref{equ:map1} to \eqref{equ:map3}
combined as the number of terms produced in $\calP'(\bfs,n)$ and
$\calP((\bfs,n))$ both follow the Fibonacci rule.
This completes the proof of the proposition.
\end{proof}

\section{Main Results}\label{sec:main}
The notation in the proceeding section is still in force.
Throughout this section we fix $\bfs'=(s_1,\dots,s_k)\in \N^k$,
$\gk:={k+1}$, $z=s_\gk\in \C$ and $\bfs=(s_1,\dots,s_\gk)$. For any subset
$\bfi=(i_1,\dots,i_t)\subseteq [k]$ we write $\bfs(\bfi)=(s_{i_1},\dots,s_{i_t}).$
For any real number $\ga$ we define
$$S(\bfs,\bfi,\ga):=\sum_{\substack{m_1,\dots,m_\gk\in \N^\gk\\
\sum_{j\in \bfi} m_j=\sum_{j\in [\gk]\setminus \bfi } m_j } }
\frac{e(m_\gk \ga)}{m_1^{s_1}\cdots m_\gk^{s_\gk} }.$$
Observe that if $i\le k$ then $S(\bfs,\{i\},\ga)$ is a
colored 1-MTZV with only variable $s_i$ dressed with $e(\ga)$
while  $S(\bfs,[k],\ga)$ is a
colored \emph{special} 1-MTZV with only variable $z$ dressed with $e(\ga)$.
This observation and the next proposition is crucial
to prove Theorem~\ref{thm:main}.
\begin{prop}\label{prop:key}
Let $\ga\in\R$ and $\emptyset\ne \bfi\subseteq [k]$. Suppose $\Re(z)\ge 1$ then we have
\begin{equation}\label{equ:tthline}
\lim_{N\to \infty} \int_0^1 \prod_{j\in \bfi} f_{s_j,N}(x)
\prod_{j\in [k]\setminus \bfi} f^{+}_{s_j,N}(x) f^{+}_{z,N}(x+\ga)\, dx=
\sum_{\bfj\subseteq \bfi} (-1)^{|\bfs(\bfj)|} S(\bfs,\bfj,\ga).
\end{equation}
If $2\le \bfi\ne [k]$ then it is a $\Q$-linear combination of
products of Riemann zeta values at a non-negative even integers
and a colored 1-MTZV with only the complex variable $z$ being dressed
with $e(\ga)$. This linear combination
is explicitly given by
\begin{equation} \label{equ:redBunch}
E(\bfs,\bfi,\ga)=
 \sum_{\bfP\in \calP'(\bfs(\bfi))} \sum_{\bfr\in\ind'(\bfP)}
  (-1)^{|\bfs(\bfi)|} 2^{\ell(\bfi)-q}
\prod_{j=1}^q  \Bigg\{\prod_{i=2}^{\ell(\bfr_j)+1}
   c_{j,i}(\bfP,\bfr)\Bigg\} C_{j,\ga}(\bfP,\bfr)
\end{equation}
where $c$'s and $C_{j,\ga}$'s are define as follows.
For $\bfs,\bfi,\ga$ as above and any positive integer $n$ let
$f_{\bfs,\bfi,\ga}(n):=\zeta_\MT\big(\bfs([\gk]\setminus\bfi),n;
    R_{k-\ell(\bfi)+1}(\ga,\{1\}_{k-\ell(\bfi)+2}) \big).$
Then
\begin{equation}\label{equ:tBx'sga}
   C_{j,\ga}(\bfP,\bfr)=\left\{
     \begin{array}{ll}
       {\displaystyle (-1)^{s_{j,l_j}}\tgz(|\bfP_j|-2|\bfr_j|), } & \hbox{if $j<q$;} \\
       {\displaystyle  f_{\bfs,\bfi,\ga}(|\bfP_q|),
         \phantom{ \frac{A}{B}}}& \hbox{if $j=q$ and $l_q=1$;} \\
      {\displaystyle f_{\bfs,\bfi,\ga}(|\bfP_q|-2|\bfr_q|),
            \phantom{ \frac{A}{B}}} & \hbox{if $j=q$ and $l_q>1$.}
     \end{array} \right.
\end{equation}
where $\tgz(m)=\zeta(m)$ if $m$ is even and $\tgz(m)=0$ if $m$ is odd.
If $l(\bfr_j)=0$ then the innermost product is $1$; otherwise
\begin{equation}\label{equ:tb'sga}
c_{j,i}(\bfP,\bfr)= \bigg[{\gs_i(\bfP_j)-2\gs_{i-1}(\bfr_j)-1 \choose s_{j,i}-1}
       +{\gs_i(\bfP_j)-2\gs_{i-1}(\bfr_j)-1 \choose s_{j,i}-2r_{j,i-1}} \bigg]
    \zeta(2 r_{j,i-1}).
\end{equation}
\end{prop}
\begin{proof} The equation \eqref{equ:tthline} is straightforward.
So we only need to prove the second part. In the following
proof we often exchange limits without giving explicit justification.
But they are easy to check by Lebesgue's Dominated Convergence
Theorem because of the absolution convergence
to be proved in Prop.~\ref{prop:conv}.

Assume $\ell(\bfi)\ge 2$. First, by Lemma~\ref{lem:Bernpoly} we have
\begin{equation}\label{equ:lim=S}
 \lim_{N\to \infty} \prod_{j\in \bfi} f_{s_j,N}(x) =
   \frac{ (2\pi i)^{|\bfs(\bfi)|}}{(-1)^{\ell(\bfi)} \cdot \bfs(\bfi)!} B_{\bfs(\bfi)}(x).
\end{equation}
Prop.~\ref{prop:BernProdNice} now yields (with the same notation given there)
\begin{multline*}
 \text{LHS of }\eqref{equ:tthline}
=  \frac{ (2\pi i)^{|\bfs(\bfi)|}}{(-1)^{\ell(\bfi)} \cdot \bfs(\bfi)!}
 \sum_{\bfP\in \calP'(\bfs(\bfi))} \sum_{\bfr\in\ind'(\bfP)}
  \Bigg\{ \prod_{j=1}^{q-1} \Bigg\{\Bigg\{\prod_{i=1}^{\ell(\bfr_j)}
   b_{j,i}(\bfP,\bfr)\Bigg\} B_{j}(\bfP,\bfr) \Bigg\}\\
   \cdot \prod_{i=1}^{\ell(\bfr_q)}
   b_{q,i}(\bfP,\bfr)  \cdot \int_0^1 B_q(\bfP,\bfr,x) \cdot \prod_{j\in [k]\setminus \bfi}
    f^{+}_{s_j,\infty}(x) f^{+}_{z,\infty}(x+\ga) \, dx  \Bigg\}
\end{multline*}
This is equal to
\begin{equation} \label{equ:redBunch**}
E(\bfs,\bfi,\ga):=\frac{ (2\pi i)^{|\bfs(\bfi)|}}{(-1)^{\ell(\bfi)} \cdot \bfs(\bfi)!}
 \sum_{\bfP\in \calP'(\bfs(\bfi))} \sum_{\bfr\in\ind'(\bfP)}
   \prod_{j=1}^{q} \Bigg\{\Bigg\{\prod_{i=1}^{\ell(\bfr_j)}
   b_{j,i}(\bfP,\bfr)\Bigg\} B_{j}(\bfP,\bfr) \Bigg\}.
\end{equation}
Here by \eqref{equ:b's} and \eqref{equ:Bernzeta}
if $l(\bfr_j)\ge 1$ then $b_{j,i}(\bfP,\bfr)=$
\begin{equation*}
 \bigg[{\gs_i(\bfP_j)-2\gs_{i-1}(\bfr_j)\choose 2r_{j,i}} s_{j,i+1}
       +{s_{j,i+1}\choose
            2r_{j,i}}\big(\gs_i(\bfP_j)-2\gs_{i-1}(\bfr_j)\big)\bigg]
    \frac{-2 (2 r_{j,i})! \zeta(2 r_{j,i})/(2\pi i)^{2 r_{j,i}}}
            {\gs_{i+1}(\bfP_j)-2\gs_i(\bfr_j)}.
\end{equation*}
By \eqref{equ:B's} and \eqref{equ:Bernzeta} if $j<q$ then
\begin{equation*}
  B_{j}(\bfP,\bfr)= (-1)^{s_{j,l_j}}(|\bfP_j|-s_{j,l_j}-2|\bfr_j|)!(s_{j,l_j}) !
        \frac{2\cdot\tgz(|\bfP_j|-2|\bfr_j|)} {(-2\pi i)^{|\bfP_j|-2|\bfr_j|}}
\end{equation*}
where $s_{j,l_j}$ is the last component of $\bfP_j$. Finally , by
\eqref{equ:Bx's} and  \eqref{equ:BernInt}
\begin{equation*}
B_q(\bfP,\bfr) =\left\{
    \begin{array}{ll}
      {\displaystyle -\frac{(|\bfP_q|)!}{(-2\pi i)^{|\bfP_q|} } \cdot\zeta_\MT(\bfs([\gk]\setminus \bfi),|\bfP_q|;\{1\}_{k-\ell(\bfi)},\ga,1),} & \hbox{if $l_q=1$;} \\
      {\displaystyle  -\frac{(|\bfP_q|-2|\bfr_q|)!}{(-2\pi i)^{|\bfP_q-2|\bfr_q||} } \zeta_\MT(\bfs([\gk]\setminus \bfi),|\bfP_q|-2|\bfr_q|;\{1\}_{k-\ell(\bfi)},\ga,1),} & \hbox{if $l_q>1$.}
    \end{array}
  \right.
\end{equation*}
We readily see that all the powers of $2\pi i$ cancel out
in \eqref{equ:redBunch**} and \eqref{equ:redBunch} has the
correct sign and $2$-powers.
Let's compute the $j$th term of the innermost product
in \eqref{equ:redBunch**}  when $j<q$.
The case $j=q$ is very similar and is left to the
interested reader.
For simplicity let us further assume $\bfP_j=(a_1,\dots,a_l)$,
$\bfr_j=(r_1,\dots,r_{l-2})$, and $l\ge 3$. Without
the signs, $2$-powers and the
zeta factors this product looks as follows:
\begin{align*}
  &\bigg[{a_1\choose 2r_1} a_2 +{a_2 \choose 2r_1}a_1\bigg]
    \frac{(2r_1)!}{a_1+a_2-2r_1} \\
  \times& \bigg[{a_1+a_2-2r_1\choose 2r_2} a_3 +{a_3 \choose 2r_2}(a_1+a_2-2r_1)\bigg]
    \frac{(2r_2)!}{a_1+a_2+a_3-2(r_1+r_2)} \\
  &\hskip2cm  \vdots \\
  \times& \bigg[{\sum_{t=1}^{l-2} a_t-2\sum_{t=1}^{l-3} r_t \choose 2r_{l-2}} a_{l-1}
    +{a_3 \choose 2r_2}\big(\sum{}{{\scriptstyle l-2}\atop{\scriptstyle t=1}} a_t
        -2\sum{}{{\scriptstyle l-3}\atop{\scriptstyle t=1}} r_t\big)\bigg]
    \frac{(2r_{l-2})!}{\sum_{t=1}^{l-1} a_t-2\sum_{t=1}^{l-2} r_t} \\
 \times&  \frac{(\sum_{t=1}^{l-1} a_t-2\sum_{t=1}^{l-2} r_t)!}{a_1!a_2!\cdots a_{l-1}!}
\end{align*}
Now multiplying each fraction like $1/(a_1+a_2-2r_1)$ on the next $[\cdots]$,
expanding the binomial coefficients,  and canceling all $(2r_t)!$'s and $a_t!$'s we get:
\begin{align*}
  &\bigg[\frac{1}{(a_1-2r_1)!(a_2-1)!}  +\frac{1}{(a_2-2r_1)!(a_1-1)!}\bigg] \\
  \times& \bigg[\frac{(a_1+a_2-2r_1-1)!}{(a_1+a_2-2(r_1+r_2))!(a_3-1)!}
     +\frac{1}{(a_3-2r_2)!}  \bigg]  \\
  &\hskip2cm  \vdots \\
 \times & \bigg[\frac{(\sum_{t=1}^{l-2} a_t-2\sum_{t=1}^{l-3})!}
        {(\sum_{t=1}^{l-2} a_t-2\sum_{t=1}^{l-2})!(a_{l-1}-1)!}
    + \frac{1}{(a_{l-1}-2r_{l-2})!} \bigg] \\
 \times&  (\sum{}{{\scriptstyle l-1}\atop{\scriptstyle t=1}} a_t-
        2\sum{}{{\scriptstyle l-2}\atop{\scriptstyle t=1}} r_t-1)!
\end{align*}
Dividing the first numerator appearing in each $[\cdots]$
(including the last line) and then multiplying
it on the $[\cdots]$ immediately proceeding it we finally
arrive at the displayed formula \eqref{equ:redBunch},
as desired.
This finishes the proof of the proposition.
\end{proof}
\begin{rem}
Clearly every $\zeta_\MT$ in \eqref{equ:tBx'sga}
is a colored 1-MTZV
with only variable $z=s_\gk$ dressed with $e(\ga)$.
Further, the depth of colored 1-MTZV is $\gk-\ell(\bfi)\le k-1$
since we assumed $\ell(\bfi)\ge 2$.
Also notice that in \eqref{equ:tBx'sga}
$\gk\in [\gk]\setminus \bfi$ for all $\bfi\in [k]$
so the 1-MTZVs are not special and therefore we cannot
use Theorem~\ref{thm:coloredMTZV=MZVs} to reduce
\eqref{equ:redBunch} further to colored 1-MZVs.
\end{rem}

\begin{thm} \label{thm:mainQuantify}
Let $k$ be a positive integer $\ge 2$ and $\bfs=(s_1,\dots,s_k)\in \N^k$.
Let $\chi$ be a primitive Dirichlet character. Then
\begin{multline}   \label{equ:keyrelQuantify}
(-1)^{k+|\bfs|}L_\MT(\bfs,z;\{\bfone\}_k,\chi)+\sum_{j=1}^k
(-1)^{s_j}L_\MT(R_j(z,\bfs),s_j;R_j(\chi,\{\bfone\}_{k+1}))\\
=\sum_{\bfi\subseteq [k],\ell(\bfi)\ge 2} (-1)^{\ell(\bfi)} E(\bfs,\bfi,\chi)
\end{multline}
for all $z\in \C$ except at singular points
where with the notation in Prop.~\ref{prop:BernProdNice}
\begin{equation} \label{equ:BernProdNiceQuantify}
E(\bfs,\bfi,\chi)=
 \sum_{\bfP\in \calP'(\bfs(\bfi))} \sum_{\bfr\in\ind'(\bfP)}
  (-1)^{|\bfs(\bfi)|} 2^{\ell(\bfi)-q}
\prod_{j=1}^q  \Bigg\{\prod_{i=2}^{\ell(\bfr_j)+1}
   c_{j,i}(\bfP,\bfr)\Bigg\} C_{j,\chi}(\bfP,\bfr)
\end{equation}
where $c$'s are defined by \eqref{equ:tb'sga}
and $C_{j,\chi}$'s are define as follows.
For $\bfs,\bfi,\chi$ as above and any positive integer $n$ let
$f_{\bfs,\bfi,\chi}(n):=L_\MT\big(\bfs([\gk]\setminus\bfi),n;
    R_{k-\ell(\bfi)+1}(\chi,\{\bfone\}_{k-\ell(\bfi)+2}) \big).$
Then $C_{j,\chi}(\bfP,\bfr)$ can be obtained by replacing
$\ga$ by $\chi$ in \eqref{equ:tBx'sga}.
\end{thm}
\begin{proof}
First we assume that $\Re(z) \ge 1$. For any $\ga\in \R$ we can
take all possible subset $\bfi$ of $[k]$ of length $t$
and add \eqref{equ:tthline} together for all these $\bfi$'s.
Within this sum we find that
each $S(\bfs,\bfj,\ga)$ with length $\ell(\bfj)=r\le t$
appears exactly ${k-r\choose t-r}$ times. Since for every fixed $r$, $1\le r\le k-1$
$$\sum_{t=r}^k {k-r\choose t-r} (-1)^t=0$$
we see that
\begin{equation}\label{equ:move}
  \sum_{\bfi\subseteq [k]} (-1)^{\ell(\bfi)}
\sum_{\bfj\subseteq \bfi} (-1)^{|\bfs(\bfj)|} S(\bfs,\bfj,\ga)
 =(-1)^{k+s_1+\cdots+s_k}\zeta_\MT(\bfs;\{1\}_,\ga).
\end{equation}
because the only term with $r=k$ is when $\bfi=\bfj=[k]$.
Further, if we take $\ell(\bfi)=1$ in \eqref{equ:move} then we get
$$-\sum_{j=1}^k (-1)^{s_j} S(\bfs,\{j\},\ga)=
 -\sum_{j=1}^k (-1)^{s_j} \zeta_\MT(R_j(z,\bfs),s_j;R_j(\ga,\{1\}_{k+1})).$$
Moving these terms from the LHS of \eqref{equ:move} to the RHS we have:
\begin{multline}  \label{equ:coloredKeyss}
 (-1)^{k+s_1+\cdots+s_k}\zeta_\MT(\bfs;\{1\}_k,\ga)
+\sum_{j=1}^k (-1)^{s_j} \zeta_\MT(R_j(z,\bfs),s_j;R_j(\ga,\{1\}_{k+1}))\\
=\sum_{\bfi\subseteq [k],\ell(\bfi)\ge 2} (-1)^{\ell(\bfi)}
\sum_{\bfj\subseteq \bfi} (-1)^{|\bfs(\bfj)|} S(\bfs,\bfj,\ga)
=\sum_{\bfi\subseteq [k],\ell(\bfi)\ge 2} (-1)^{\ell(\bfi)} E(\bfs,\bfi,\ga)
\end{multline}
by Prop.~\ref{prop:key}.
Note that for any $j\le k$ and any primitive Dirichlet character $\chi$
of conductor $f$ by \cite[Lemma~4.7]{Wash} we have
\begin{equation}\label{equ:SisMT}
L_\MT(R_j(z,\bfs),s_j;R_j(\chi,\{\bfone\}_\gk))
=\sum_{n=1}^{f}  \frac{\chi(n)}{\tau(\ol{\chi})}
    \zeta_\MT(R_j(z,\bfs),s_j;R_j( f/n,\{1\}_{k+1})).
\end{equation}
Replacing $\ga$ in \eqref{equ:coloredKeyss} by $n/f$, multiplying
by $\chi(n)/\tau(\ol{\chi})$, and summing over
$n=1,\dots,f$ we finally arrive at \eqref{equ:keyrel}.
We notice that the theorem is now proved under the assumption
that $\Re(z)\ge 1$. But we can easily remove this restriction
by analytic continuation using Theorem~\ref{thm:anaCont}.
This completes the proof of Theorem~\ref{thm:mainQuantify}.
\end{proof}

\begin{rem}\label{rem:genForm}
We can compute the LHS of \eqref{equ:tthline} explicitly in a
not too complicated form by Prop.~\ref{prop:BerProd}.
In \eqref{equ:BernInt} taking $n=n(\bfs,\bfi,\bfj,\bfP):=
|\bfs(\bfi)|-|\bfj|+\ell(\bfv)-\ell(\bfi)+1$ and
using \eqref{equ:lim=S} we have
\begin{multline*}
 \lim_{N\to \infty} \int_0^1 \prod_{j\in \bfi} f_{s_j,N}(x)
\prod_{j\in [\gk]\setminus \bfi} f^{+}_{s_j,N}(x) \, dx=
  \frac{ (2\pi i)^{|\bfs(\bfi)|}}{(-1)^{\ell(\bfi)} \bfs(\bfi)!} \cdot\\
 \left\{ \sum_{ \bfv\subsetneq [\ell(\bfi)] }
\sum_{0\le j_\bfv\le s_\bfv}
 {n-1\choose  \bfs(\bfi) -\Inf_\bfv^{\ell(\bfi)}(\bfj)} \frac{B_\bfj}{\bfj!}
    \frac{\bfs(\bfi)! }{n} \int_0^1 B_n(x)
    \prod_{j\in [\gk]\setminus \bfi} f^{+}_{s_j,\infty}(x)\, dx \right\}\\
= -\sum_{ \bfv\subsetneq [\ell(\bfi)] }
\sum_{0\le j_\bfv\le s_\bfv}
  \frac{ (2\pi i)^{|\bfs(\bfi)|-n} n!}{(-1)^{\ell(\bfi)+n}\bfs(\bfi)!}
 {n-1 \choose  \bfs(\bfi) -\Inf_\bfv^{\ell(\bfi)}(\bfj)}
 \frac{B_\bfj}{\bfj!}
\frac{\bfs(\bfi)! }{n} \zeta_\MT\big(\bfs([\gk]\setminus \bfi),n\big).
\end{multline*}
Note however,
this formula is not enough to prove Theorem~\ref{thm:main}.
\end{rem}

\section{Some Corollaries and Examples}

Because of the appearance of the odd powers of $2\pi i$
we don't get explicitly reduced form of \eqref{equ:keyrel}
by using computations contained in Remark~\ref{rem:genForm}.
We have to use the more involved Theorem~\ref{thm:mainQuantify}.
In Theorem~\ref{thm:mainQuantify} taking $k=2$ we immediately get
\begin{cor}\label{cor:depth2} Let $a,b\in \N$ and $\chi$ be any
primitive Dirichlet character. Then
\begin{multline}   \label{equ:dept2Zag}
L_\MT(a,b,z;\bfone,\bfone,\chi)+
(-1)^b L_\MT(z,b,a;\chi,\bfone,\bfone) +
(-1)^a L_\MT(a,z,b;\bfone,\chi,\bfone)\\
=2\sum_{r=0}^{\lfloor\max\{a,b\}/2\rfloor}
\bigg[{a+b-2r-1 \choose a-1}+
    {a+b-2r-1 \choose a-2r} \bigg]\zeta(2 r)L(a+b+z-2r;\chi)
\end{multline}
for all complex number $z\in \C$ except at singular points.
\end{cor}
This is in agreement with \cite[Prop.~2.2]{MNT} by Matsumoto et al.
Note also that Tsumura's result \cite[Theorem~4.5]{Tsu2} should reduce to
Cor.~\ref{cor:depth2} with $\chi=\bfone$ (see \cite[Theorem~1.2]{Nak}
and its remarks).

The depth $d=3$ case is essentially the same as that of \cite[Theorem~3.5]{MNT}.
We now look at depth $d=4$. For any function $F(x_1,\dots,x_n)$ we define
$$\per{x_1,\dots,x_n}F(x_1,\dots,x_n)=
\sum_{j=1}^{n} F( x_1,\dots,\widehat{x_j},\dots,x_n ,x_j).$$
The following lemma will be used when we need to show strong reducibility result.
\begin{lem} \label{lem:strongRed}
For any positive integers $a,b,c$ and $d$ we have
$$\zeta_\MT(a,b,c)=\per{a,b}\sum_{\nu=0}^{b-1}
 {a+\nu-1\choose \nu}\zeta(c+a+\nu,b-\nu).$$
and
\begin{multline*}
\zeta_\MT(a,b,c,d)=\per{a,b,c}\sum_{\nu_1=0}^{a-1}\sum_{\nu_2=0}^{b-1}
 {\nu_1+\nu_2+c-1\choose \nu_1,\nu_2,c-1} \\
 \left\{\sum_{\nu_3=0}^{a-\nu_1-1}
 {b-\nu_2+\nu_3-1\choose \nu_3}
 \zeta(c+d+\nu_1+\nu_2,b-\nu_2+\nu_3,a-\nu_1-\nu_3)\right.\\
\left. +
 \sum_{\nu_3=0}^{b-\nu_2-1}
 {a-\nu_1+\nu_3-1\choose \nu_3}\zeta(c+d+\nu_1+\nu_2,b-\nu_2-\nu_3,a-\nu_1+\nu_3).\right\}
\end{multline*}
\end{lem}
\begin{proof} See the proof of \cite[Theorem~5]{ZB}.
\end{proof}

As the signed cyclic sum formula \eqref{equ:keyrel} has been derived in depth 3
in \cite{MNT} we provide the depth $4$ expression explicitly below.
\begin{cor} \label{cor:specialdepth=4}
Let $\bfs=(s_1,s_2,s_3,s_4)=(a,b,c,d)\in\N^4$. The signed cyclic sum of colored 1-MTZVs
\begin{multline*}
(-1)^{|\bfs|} \zeta_\MT(\bfs,z;\{1\}_4,\ga)
+\sum_{j=1}^4(-1)^{s_j} \zeta_\MT(R_j(z,\bfs),s_j;R_j(\ga,\{1\}_4)) \\
= \sum_{1\le i<j\le 4} E_2(\bfs,(i,j),\ga)-
 \sum_{1\le i<j<k\le 4} E_3(\bfs,(i,j,k),\ga)+E_4(\bfs,(1,2,3,4),\ga)
\end{multline*}
is reducible for $z\in \C$ except for singular points, where
$E_\ell$ ($\ell=2,3,4$) are defined
defined as follows: $E_\ell(\bfs,(i_1,\dots,i_\ell),\ga)=
\pi_{\left({{1,\dots,t} \atop {i_1 ,\dots,i_\ell}}\right)} E_\ell(\bfs,(1,\dots,\ell),\ga)$
(permuting the $s_j$'s) and
$$
\frac{E_2(\bfs,(1,2),\ga)} { 2 (-1)^{a+b} }
 =\per{a,b}\sum_{r=0}^{\lfloor \max\{a,b\}/2\rfloor} {a+b-2r-1 \choose b-1}
    \zeta(2r)\zeta_\MT(c,d,z,a+b-2r;1,1,\ga,1),$$
\begin{align*}
\ &\frac{E_3(\bfs,(1,2,3),\ga)} { 2 (-1)^{a+b+c}}=(-1)^b
  \tgz(a+b)\zeta_\MT(d,z,c;1,\ga,1) \\
&\quad +2\cdot \per{a,b}\sum_{\mu=0}^{\lfloor \max\{a,b\}/2\rfloor}
   \sum_{\nu=0}^{\lfloor \max\{a+b-2\mu,c\}/2\rfloor}\left[
{a+b-2\mu-1 \choose b-1}{a+b+c-2\mu-2\nu-1 \choose c-1}\right.\\
    \ &\hskip1cm    \left.+{a+b+c-2\mu-2\nu-1 \choose a-2\mu,\ b-1,\ c-2\nu} \right]
    \zeta(2\mu)\zeta(2\nu)\zeta_\MT(d,z,a+b+c-2\mu-2\nu;1,\ga,1)
\end{align*}
where $\tgz(n)=\zeta(n)$ if $n$ is even and $\tgz(n)=0$ if $n$ is odd.
Finally, if $\bfi=[4]$ of length $4$ then setting $\gs=a+b+c+d$
we can get by \eqref{equ:redBunch}
\begin{align*}
 &\ \frac{E_4(\bfs,(1,2,3,4),\ga)} { 4 (-1)^{a+b+c+d} }
 =2 \cdot \per{a,b}\sum_{\mu=0}^{\lfloor \max\{a,b\}/2\rfloor}\
    \sum_{\nu=0}^{\lfloor \max\{a+b-2\mu,c\}/2\rfloor}
    \ \sum_{\gl=0}^{\lfloor \max\{a+b+c-2\mu-2\nu,d\}/2\rfloor} \\
\ & \hskip1cm  \left[{a+b-2\mu-1 \choose b-1}{a+b+c-2\mu-2\nu-1 \choose c-1}
    {\gs-2\mu-2\nu-2\gl-1 \choose d-1}\right.\\
\ & \hskip1cm  +{a+b+c-2\mu-2\nu-1 \choose a-2\mu,\ b-1,\ c-2\nu}
        {\gs-2\mu-2\nu-2\gl-1 \choose d-1} \\
\ & \hskip1cm +{a+b-2\mu-1 \choose b-1}
    {\gs-2\mu-2\nu-2\gl-1 \choose a+b-2\mu-2\nu,\ c-1,\ d-2\gl} \\
\ & \hskip1cm \left.+{\gs-2\mu-2\nu-2\gl-1 \choose a-2\mu,\ b-1,\ c-2\nu,\ d-2\gl}\right]
  \cdot\zeta(2\mu)\zeta(2\nu)\zeta(2\gl)\zeta_\MT(z,\gs-2(\mu+\nu+\gl);\ga,1)\\
&  +(-1)^b \tgz(a+b)  \cdot\per{c,d}\sum_{\mu=0}^{\lfloor \max\{c,d\}/2\rfloor} {c+d-2\mu-1 \choose d-1}
         \zeta(2\mu) \zeta_\MT(z,c+d-2\mu;\ga,1)  \\
&+(-1)^c \per{a,b}\sum_{\mu=0}^{\lfloor \max\{a,b\}/2\rfloor}{a+b-2\mu-1 \choose b-1}
         \zeta(2\mu)\tgz(a+b+c-2\mu)\zeta_\MT(z,d;\ga,1)
\end{align*}
\end{cor}
\begin{proof} This is follows from Theorem~\ref{thm:mainQuantify} easily.
\end{proof}

When $a=b=c=d=n$ setting $E_j=E(\bfs,\{n\}_j,1)$ for $j=2,3,4$ we have
\begin{align}
E_2 =&4\sum_{r=0}^{\lfloor n/2\rfloor} {2n-2r-1 \choose n-1}
    \zeta(2r)\zeta_\MT(n,n,z,2n-2r;1,1,\ga,1), \label{equ:E2}\\
E_3=  & 2\zeta(2n)\zeta_\MT(n,z,n) +8(-1)^n\cdot \sum_{\mu=0}^{\lfloor n/2\rfloor}
   \sum_{\nu=0}^{\lfloor \max\{2n-2\mu,n\}/2\rfloor}\left[
{2n-2\mu-1 \choose n-1}{3n-2\mu-2\nu-1 \choose n-1}\right.\notag \\
    \ &\hskip1cm    \left.+{3n-2\mu-2\nu-1 \choose n-2\mu,\ n-1,\ n-2\nu} \right]
    \zeta(2\mu)\zeta(2\nu)\zeta_\MT(n,z,3n-2\mu-2\nu;1,\ga,1) \label{equ:E3}\\
E_4 =  & 16\cdot  \sum_{\mu=0}^{\lfloor n/2\rfloor}\
    \sum_{\nu=0}^{\lfloor \max\{2n-2\mu,n\}/2\rfloor}
    \ \sum_{\gl=0}^{\lfloor \max\{3n-2\mu-2\nu,n\}/2\rfloor} \notag \\
\ & \  \left[{2n-2\mu-1 \choose n-1}{3n-2\mu-2\nu-1 \choose n-1}
    {4n-2\mu-2\nu-2\gl-1 \choose n-1}\right.\notag \\
\ & \  +{3n-2\mu-2\nu-1 \choose n-2\mu,\ n-1,\ n-2\nu}
        {4n-2\mu-2\nu-2\gl-1 \choose n-1} \notag \\
\ & \  +{2n-2\mu-1 \choose n-1}
    {4n-2\mu-2\nu-2\gl-1 \choose 2n-2\mu-2\nu,\ n-1,\ n-2\gl} \notag \\
\ & \  \left.+{4n-2\mu-2\nu-2\gl-1 \choose n-2\mu,\ n-1,\ n-2\nu,\ n-2\gl}\right]
  \cdot\zeta(2\mu)\zeta(2\nu)\zeta(2\gl)\zeta(z+4n-2(\mu+\nu+\gl);\ga)\notag \\
&  +(-1)^n 8  \zeta(2n)  \cdot \sum_{\mu=0}^{\lfloor n/2\rfloor} {2n-2\mu-1 \choose n-1}
         \zeta(2\mu) \zeta(z+2n-2\mu;\ga)  \notag \\
&+8 \sum_{\mu=0}^{\lfloor n/2\rfloor}{2n-2\mu-1 \choose n-1}
         \zeta(2\mu)\tgz(3n-2\mu)\zeta(z+n;\ga) \label{equ:E4}
\end{align}
Thus we get
\begin{cor} \label{cor:specialn}
Let $\ga\in \R$ and $n\in \N$. Then signed sum of colored 1-MTZVs
\begin{equation}  \label{equ:spdept=4}
 \zeta_\MT(\{n\}_4,z;\{1\}_4,\ga)-4\zeta_\MT(\{n\}_3,z,n;\{1\}_3,\ga,1)
    = 6E_2-4E_3+E_4
\end{equation}
for all $z\in \C$ except at singular points, where $E_2$,  $E_3$, and  $E_4$
are defined by \eqref{equ:E2},  \eqref{equ:E3}, and  \eqref{equ:E4}, respectively.
\end{cor}
\begin{proof} This follow from Cor.\ref{cor:specialdepth=4} since
$$\text{LHS of } \eqref{equ:spdept=4}={4\choose 2}E_2
 -{4\choose 3}E_3 +E_4  .$$
\end{proof}

\begin{cor} \label{cor:special1}
Let $\ga\in \R$. For all $z\in \C$ the signed sum of 1-MTZVs
\begin{multline}\label{equ:z=n}
4\zeta_\MT(\{1\}_3,z,1;\{1\}_3,\ga,1)-\zeta_\MT(\{1\}_4,z;\{1\}_4,\ga)=
12\zeta_\MT(1,1,z,2;\{1\}_3,\ga,1)\\
+24\big[\zeta(2)\zeta_\MT(1,z,1;1,\ga,1)-\zeta_\MT(1,z,3;1,\ga,1)-
\zeta(2)\zeta(z+2;\ga)+\zeta(z+4;\ga)\big],
\end{multline}
except at singular points.
\end{cor}
\begin{proof}
Specializing \eqref{equ:E2} to \eqref{equ:E4} further to $n=1$ we get:
\begin{align*}
E_2=&-2\zeta_\MT(1,1,z,2;1,1,\ga,1),\\
E_3=&6(\zeta(2)\zeta_\MT(1,z,1;1,\ga,1)-\zeta_\MT(1,z,3;1,\ga,1)),\\
E_4=&24(\zeta(2)\zeta_\MT(z+2;\ga)-\zeta_\MT(z+4;\ga)).
\end{align*}
So the corollary follows from Cor.~\ref{cor:specialn} at once.
\end{proof}

\begin{cor} \label{cor:specialFurthern}
For all $n\in \N$ the signed cyclic sum of MTZVs
\begin{multline*}
4\zeta_\MT(\{1\}_3,n,1)-\zeta_\MT(\{1\}_4,n)=
12\Bigg\{
 2\zeta(n+4)-2\zeta(n+3,1)+2\zeta(n+2,1,1)\\
+2\zeta(2)(\zeta(n+1,1)-\zeta(n+2))+\sum_{\nu=0}^{n-1}\sum_{\mu=0}^{n-1-\nu}
    \big[\zeta(3+\nu,1+\mu,n-\nu-\mu)+\zeta(3+\nu,n-\nu-\mu,1+\mu)\big]\\
+\sum_{\nu=0}^{n-1} \big[2\zeta(2)\zeta(2+\nu,n-\nu)
    -2\zeta(4+\nu,n-\nu)+\zeta(3+\nu,1,n-\nu)+\zeta(3+\nu,n-\nu,1)\big]\Bigg\}
\end{multline*}
is strongly reducible.
\end{cor}
\begin{proof}
The corollary follows from Lemma~\ref{lem:strongRed}
after specializing $z=n$ and $\ga=1$ in \eqref{equ:z=n}.
\end{proof}

\begin{rem} Corollary \ref{cor:specialFurthern} is consistent with
\cite[Cor.~4.2]{Hoff} when we take $n=1$. In fact
both sides are then equal to $72\zeta(5)$ since one can show
by double shuffle relations of MVZs that $\zeta(4,1)=\zeta(3,1,1)$
and $\zeta(4,1)+\zeta(2)\zeta(3)=2\zeta(5)$ (see \cite[\S3]{Zesum}).
\end{rem}

Maple computation by Cor.~\ref{cor:specialFurthern} also shows that
\begin{multline*}
 \zeta_\MT(\{2\}_5)=\frac{12}5\Big\{2\zeta(2)\zeta_\MT(2,2,2,2)
    +24\zeta(2)\zeta_\MT(2,2,4)-10\zeta(4)\zeta_\MT(2,2,2)\\
 -30\zeta_\MT(2,2,6)-3\zeta_\MT(2,2,2,4)\Big\}+2\zeta(10)
\end{multline*}
is reducible, which, by Lemma \ref{lem:strongRed}, can be further reduced to
\begin{multline*}
2\zeta(10) -\frac{216}5\Big\{\zeta(6,2,2)+2\zeta(6,3,1)+2\zeta(7,1,2)+4\zeta(7,2,1)+6\zeta(8,1,1)\Big\} \\ - 144\Big\{\zeta(8,2)+2\zeta(9,1)\Big\}-48\zeta(4)\Big\{6\zeta(4,2)+2\zeta(5,1)\Big\} \\
+\frac{144}5\zeta(2)\Big\{ \zeta(4,2,2)+2\zeta(4,3,1)+2\zeta(5,1,2)+4\zeta(5,2,1)+6\zeta(6,1,1)+4\zeta(6,2)+8\zeta(7,1)\Big\}.
\end{multline*}
By parity consideration \cite[Cor.~8]{IKZ} we know that the above
can be reduced further to products of MZVs of depth one or two.
In fact, after using double shuffle relations we find
\begin{align*}
 \zeta_\MT(\{2\}_5)=&\frac{79}5\zeta(10)+12 \Big\{15\zeta(8,2)+30\zeta(9,1)
    -12\zeta(2)\zeta(6,2) -24\zeta(2)\zeta(7,1)\\
 \ &\hskip7cm -4\zeta(4)\zeta(4,2)-8\zeta(4)\zeta(5,1))\Big\}\\
=&7\zeta(10)+36 \Big\{5\zeta(8,2)+10\zeta(9,1)-4\zeta(2)\zeta(6,2) -8\zeta(2)\zeta(7,1)\Big\}
\end{align*}
since $\zeta(4,2)+2\zeta(5,1)=\zeta(6)/6.$ We now can look up the table \cite{pet}
and get
$$\zeta_\MT(\{2\}_5)=\frac{1376}{385}\zeta(2)^5+180\Big\{\zeta(8,2)
-\zeta(5)^2-2\zeta(3)\zeta(7)\Big\}
+144\Big\{2\zeta(2)\zeta(3)\zeta(5)-\zeta(2)\zeta(6,2)\Big\}$$
\begin{rem}\label{rem:mt25}
Note that only one depth 2 weight $10$ term, namely $\zeta(8,2)$,
appears in the reduction
of $\zeta_\MT(\{2\}_5)$. By the table in \cite{pet} we know it's the only depth two
weight 10 MZV in the basis (there are only 7 $\Q$-linearly independent
MVZs of weight 10 by Zagier's conjecture). Moreover, by Broadhurst conjecture
\cite[(3)]{Br2} this depth two value cannot be reduced further,
hence  neither can $\zeta_\MT(\{2\}_5)$.
\end{rem}
We also have calculated $\zeta_\MT(\{2\}_5)$ by the method
in Theorem~\ref{thm:ZB}. Using EZ-face we find our two methods
produce the same value $\zeta_\MT(\{2\}_5)=.163501600521337009\dots$. %.16350160052133700900
Similarly we have also verified by two methods that
\begin{multline*}
 \zeta_\MT(\{2\}_6)=
1200\Big\{21\zeta(2)\zeta(5)^2+33\zeta(2)\zeta(8,2)+30\zeta(2)\zeta(3)\zeta(7)
    +12\zeta(8,2,1,1)-\zeta(3)^4\Big\}\\
+60\Big\{\frac{1056}{7}\zeta(2)^3\zeta(3)^2-4264\zeta(3)\zeta(9)
    -1068\zeta(10,2)-6627\zeta(5)\zeta(7)\Big\}\\
+7488\Big\{\zeta(5)\zeta(3)\zeta(2)^2+2\zeta(2)^2\zeta(6,2)\Big\}
+\frac{13944719168}{525525}\zeta(2)^6=.15311508886  \dots
\end{multline*}
% \zeta_\MT(\{2\}_6)=.15311508886232656656404640501021206470407
 and
$\zeta_\MT(\{3\}_6)=.01255766232\dots$.
%=.012557662320681327614012040354612688021722243629999
Note that only one depth $4$ term, namely $\zeta(8,2,1,1)$,
appears in the reduction
of $\zeta_\MT(\{2\}_6)$. By reasons similar to those
explained in Remark~\ref{rem:mt25} we see that
$\zeta_\MT(\{2\}_6)$ cannot be reduced further.

\section{Convergence Problem}
In the proof of Prop.~\ref{prop:key} we need the following results
which guarantees our exchange of limits there.
\begin{prop} \label{prop:conv} Let $k$ be a positive integer.
For all $(s_1,\dots,s_k)\in \N^k$ and
$s_{k+1}\in \C$ with $\Re(s_{k+1})\ge 1$,
the following series converges:
$$
S_\bfi(\bfs)=\sum_{\substack{m_1,\dots,m_{k+1}\in \N^{k+1}\\
\sum_{j\in \bfi} m_j=\sum_{j\in [\gk]\setminus \bfi } m_j } }
\frac{1}{m_1^{s_1}m_2^{s_2}\cdots m_{k+1}^{s_{k+1}} }.
$$
\end{prop}
\begin{proof} If  $\ell(\bfi)=1$ then this is the well known
MTZV $\zeta_\MT(\{1\}_{k+1})$ which certainly converges.
So we assume $\ell(\bfi)=t+1\ge 2$. Then we only need
to show the following series converges:
$$S:=\sum_{\substack{m_1,\dots,m_{k+1}\in \N^{k+1}\\
 m_1+\cdots+m_{t+1}=m_{t+2}+\cdots+m_{k+1}} }
\frac{1}{m_1 m_2\cdots  m_{k+1}}$$
Let $n=m_1+\cdots+m_{t+1}=m_{t+2}+\cdots+m_{k+1}$,
 $a_i=\sum_{j=1}^i m_i$, $b_i:=\sum_{j=t+2}^i m_j $. Then
$$ S:=\sum_{n=1}^\infty \sum_{
 \substack{m_1,\cdots,m_k\in \N^k,\\
m_1+ \cdots +m_t<n,\\
 m_{t+2}+ \cdots +m_k<n}}
\frac{1}{m_1 \cdots m_t (n-a_t)
 m_{t+2} \cdots m_k (n-b_k) }.$$
Repeatedly using partial fractions we have
\begin{multline}
\frac{1}{m_1 \cdots m_t (n-a_t)}= \frac{1}{m_1 \cdots m_{t-1} (n-a_{t-1})}
\left(\frac1{m_t}+\frac1{n-a_t}\right)\\
= \cdots
=\frac1n \prod_{j=1}^t\left(\frac1{m_j}+\frac1{n-a_j}\right).\label{equ:dec}
\end{multline}
Similarly
\begin{equation} \label{equ:dec2}
  \frac{1}{ m_{t+2} \cdots m_k (n-b_k)}=\frac1n \prod_{j=t+2}^{k}\left(\frac{1}{m_j}
    +\frac1{n-b_j}\right) .
\end{equation}
Note that the summation in $S$ can be written as
\begin{equation}\label{equ:Sdef}
\sum_{n=1}^\infty \sum_{m_1=1}^n\sum_{m_2=1}^{n-a_1-1}
\cdots \sum_{m_t=1}^{n-a_{t-1}-1}
 \sum_{m_{t+2}=1}^n \sum_{m_{t+3}=1}^{n-b_{t+2}-1}  \cdots
\sum_{m_k=1}^{n-b_{k-1}-1}
\end{equation}
After expanding \eqref{equ:dec} and taking the summation
like in \eqref{equ:Sdef} we see that there are $2^t$ products
each of which has factors $\sum 1/m_j$ or $\sum 1/(n-a_j)$
(but not both for each $j$) for $j=1,\dots,t$. Starting from $j=t$ down
to $j=1$ we now make change of the index
$m_j\to n-a_j$ if and only if $1/(n-a_j)$
appears in the product. Now if $j<t$ this substitution will affect
the $l$-th summation if and only if two conditions are satisfied:
(i) $l>j$ and (ii) such change of index was not carried
out for $l$-th summation, namely, it is still of the form
$\sum_{m_l=1}^{n-a_l}(1/m_j)$. The effect is to change the $l$-th
summation to $\sum_{m_l=1}^{m_j+a_j-a_l-1}(1/m_j)$. If this happens
the $l$-th summation will not change anymore under
substitutions for indices $m_1,\dots,m_{j-1}$. The upshot is, we can
bound each $j$-th summation by $2\sum_{m_j=1}^n 1/m_j$ for
$j=1,\dots,t$. After similarly treating \eqref{equ:dec2} we see
immediately that
$$ S\le \sum_{n=1}^\infty  \frac{4^{k-1}}{n^2}
    \left(\sum_{m=1}^n  \frac1{m} \right)^{k-1}
        \ll  \sum_{n=1}^\infty \frac{4^k\log^k(n)}{n^2} <\infty.$$
This finishes the proof of the proposition.
\end{proof}


\begin{thebibliography}{99} \raggedright

\bibitem{AS}
M.\ Abramowitz and I.A.\ Stegun, Handbook of Mathematical Functions with
Formulas, Graphs, and Mathematical Tables, Dover, New York, 1972.

\bibitem{AET}
S.~Akiyama, S.~Egami, and Y.~Tanigawa, \textit{Analytic
continuation of multiple zeta-functions and their values at
non-positive integers}, Acta Arith, \textbf{98} (2001), 107--116.

\bibitem{AI}
S.\ Akiyama and H.\ Ishikawa,  On analytic continuation of multiple
$L$-functions and related zeta-functions. In:  Analytic number theory
(Beijing/Kyoto, 1999), ed. by C.\ Jia and K.\ Matsumoto,   Dev. Math.~6,
Kluwer Acad. Publ., Dordrecht, 2002, pp.~1--16.

\bibitem{BJOP}
M.\ Bigotte, G.\ Jacob, N.E.\ Oussous, and M.\ Petitot,
Lyndon words and shuffle algebras for generating the
coloured multiple zeta values relations tables,
\emph{Theor.\ Comput.\ Sci.}\ \textbf{273} (2002), 271--283.

\bibitem{pet}
M.\ Bigotte, G.\ Jacob, N.E.\ Oussous, M.\ Petitot, and H.N.\ Minh
Table of the coloured zeta function, with 1 roots, up to the weight 12.
Available at \url{http://www2.lifl.fr/~petitot/}

\bibitem{EZface}
J.~Borwein, P.~Lisonek, and P.~Irvine, \emph{An interface for
evaluation of Euler sums}, available online at
\url{http://oldweb.cecm.sfu.ca/cgi-bin/EZFace/zetaform.cgi}

\bibitem{Br2}
D.J.\ Broadhurst,  \emph{Conjectured enumeration of irreducible
multiple zeta values, from knots and Feynman diagrams,}  preprint:
hep-th9612012.

\bibitem{Carlitz}
L.\ Carlitz, \emph{Note on the integral of the product of several Bernoulli
polynomials}, J. London Math. Soc. \textbf{34}(1959), 361--363.

\bibitem{DG}
P.~Deligne and A.~Goncharov, Groupes fondamentaux motiviques
de Tate mixte, \emph{Ann.\ Sci.\ de l'\'Ecole Normale
Sup\'erieure},  \textbf{38} (1)(2005), 1--56.

\bibitem{Hoff}
M.~E.~Hoffman, Multiple harmonic series,
\textit{Pacific J. Math.} \textbf{152} (1992), 275--290.

\bibitem{Hweb}
M.~E.~Hoffman, References on multiple zeta values and Euler sums,
\url{http://www.usna.edu/Users/math/meh/biblio.html}

\bibitem{IKZ}
K. Ihara, M. Kaneko, and D. Zagier, Derivation and double shuffle
relations for multiple zeta values,  \textit{Compositio Math.}
\textbf{142} (2006), 307--338.

\bibitem{Mats}
K. Matsumoto, On the analytic continuation of various
multiple-zeta functions,' in Number Theory for the Millennium
(Urbana, 2000), Vol. II, M. A. Bennett et. al. (eds.),
A. K. Peters, Natick, MA, 2002, pp.\ 417--440.

\bibitem{Mats2}
K. Matsumoto, On Mordell-Tornheim and other multiple zeta-functions,
Proc.\ of the Session in Analytic Number Theory and Diophantine Equations
(Bonn, January-June 2002), D. R. Heath-Brown and B. Z. Moroz (eds.),
Bonner Mathematische Schriften Nr. 360, Bonn 2003, No. 25, 17 pp. \textbf{MR} 2075634

\bibitem{MNOT}
K. Matsumoto, T. Nakamura, H. Ochiai, and H. Tsumura,
On value-relations, functional relations and singularities
of Mordell-Tornheim and related triple zeta-functions,
Acta Arith. \textbf{132} (2008), 99--125.

\bibitem{MNT}
K.~Matsumoto, T.~Nakamura and H.~Tsumura, Functional
relations and special values of Mordell-Tornheim triple zeta and
$L$-functions, \textit{Proc.\ Amer.\ Math.\ Soc.},  \textbf{136}(3)
(2008), 2135--2145.

\bibitem{Mord} L.J.~Mordell, On the evaluation of some multiple
series, \textit{J. London Math. Soc.}, \textbf{33} (1958), 368--271.

\bibitem{Nak} T.~Nakamura, A functional relation for the Tornheim double
zeta function, \textit{Acta Arith.} \textbf{125}(3) (2006),
257--263.

\bibitem{Nak2} T.~Nakamura,
Double Lerch series and their functional relations,
\textit{Aequationes Math.} \textbf{75} (2006),
251--259.

\bibitem{Nielsen} N.\ Nielsen, Trait\'e \'el\'ementaire des nombres
de Bernoulli,  Paris, 1923.

\bibitem{Rac}
G.~Racinet, Doubles m\'elanges des polylogarithmes multiples aux racines de
l'unit\'e, \emph{Publ.\ Math.\ IHES} \textbf{95} (2002), 185–231.

\bibitem{Stan}
R.P.\ Stanley, Enumerative Combinatorics, Vol.\ I, Wadsworth \& Brooks/Cole
Advanced Books \& Software, Monterey, California, 1986.

\bibitem{SS} M.V.\ Subbarao and R.\ Sitaramachandrarao, On
some infinite series of L.\ J.\ Mordell and their analogues,
\textit{Pacific J.\ Math.}, \textbf{119}(1) (1985),  245--255.

\bibitem{Torn} L.~Tornheim, Harmonic double series, \textit{Amer.\ J.\
Math.}, \textbf{72} (1950), 303--314.

\bibitem{Tsu1}
H.~Tsumura,Combinatorial relations for Euler-Zagier sums,
\emph{Acta Arithmetica} \textbf{111} (2004), 27--42.

\bibitem{Tsu2}
H.~Tsumura, On functional relations between the
Mordell-Tornheim double zeta functions and the Riemann zeta
functions, \textit{Math.\ Proc.\ Cambridge Philos.\ Soc.},
\textbf{142} (2007), 395--405.

\bibitem{Wash}
L.\ Washington, Introduction to Cyclotomic Fields,
Graduate texts in mathematics \textbf{83}, 2nd edition,
Springer, 1997.


\bibitem{Zag} D.~Zagier, Values of zeta
function and their applications, \textit{Proceedings of the First
European Congress of Mathematics}, \textbf{2} (1994), 497--512.


\bibitem{Zana}
J.~Zhao,  Analytic continuation of multiple zeta
functions, \emph{Proc.\ of AMS,} \textbf{128} (1999), 1275--1283.

\bibitem{Znote}
J.~Zhao,  Multiple polylogarithm values at roots of unity,
\emph{C.\ R.\ Acad.\ Sci.\ Paris}, Ser.\ I, \textbf{346} (2008), 1029-1032.

\bibitem{Zesum}
J.~Zhao,  Double shuffle relations of Euler sums,
\url{http://arxiv.org/abs/0705.2267}

\bibitem{Zpolrel}
J.~Zhao, Standard relations of multiple polylogarithm values at roots of unity.
\url{http://arxiv.org/abs/0707.1459}


\bibitem{ZB}
X.~Zhou and D.M.~Bradley,
On Mordell-Tornheim sums and multiple zeta values, submited.

\end{thebibliography}
\end{document}